\documentclass[12pt]{amsart}
\usepackage{amsmath}
\usepackage{amsthm}
\usepackage{amssymb}
\usepackage{enumerate}
\usepackage{xcolor}
\usepackage{comment}
\newtheorem{theorem}{Theorem}[section]
\newtheorem{lemma}[theorem]{Lemma}
\newtheorem{proposition}[theorem]{Proposition}
\newtheorem{corollary}[theorem]{Corollary}

\newtheorem{remar}[theorem]{Remark}
\theoremstyle{definition}
\newtheorem{example}[theorem]{Example}
\newtheorem{prob}[theorem]{Open Problem}

\newenvironment{remark}{\begin{remar}\rm}{\end{remar}}

\newcommand{\bfind}[1]{\index{#1}{\bf #1}}
\newcommand{\n}{\par\noindent}
\newcommand{\sn}{\par\smallskip\noindent}
\newcommand{\mn}{\par\medskip\noindent}
\newcommand{\bn}{\par\bigskip\noindent}
\newcommand{\pars}{\par\smallskip}
\newcommand{\parm}{\par\medskip}

\newcommand{\ovl}[1]{\overline{#1}}

\newcommand{\ina}{^{\it in}}
\newcommand{\sep}{^{\it sep}}

\newcommand{\chara}{\mbox{\rm char}\,}

\newcommand{\Gal}{\mbox{\rm Gal}\,}

\newcommand{\crf}{\mbox{\rm crf}\,}

\newcommand{\ann}{\mbox{\rm ann}\,}
\newcommand{\cO}{\mathcal{O}}
\newcommand{\cM}{\mathcal{M}}
\newcommand{\cE}{\mathcal{E}}
\newcommand{\cA}{\mathcal{A}}

\newcommand{\cC}{\mathcal{C}}
\newcommand{\cD}{\mathcal{D}}
\newcommand{\rmd}{\mbox{\rm d}\,}
\newcommand{\rme}{\mbox{\rm e}\,}
\newcommand{\rmf}{\mbox{\rm f}\,}
\newcommand{\rmg}{\mbox{\rm g}\,}

\newcommand{\Q}{\mathbb Q}
\newcommand{\N}{\mathbb N}
\newcommand{\Z}{\mathbb Z}
\newcommand{\F}{\mathbb F}

\begin{document}
%
%
\title[]{K\"ahler differentials of extensions of valuation rings and deeply ramified
fields}

\author{Steven Dale Cutkosky and Franz-Viktor Kuhlmann}
\date{2.\ 6.\ 2025}

\thanks{The first author was partially supported by grant DMS 2054394
from  NSF of the United States.}
\thanks{The second author was partially supported by Opus grant 2017/25/B/ST1/01815 
from the National Science Centre of Poland. He was also supported by a Miller Fellowship
during a four weeks visit to the University of Missouri Mathematics Department in 2022, and 
a second visit in 2024 was funded by the University of Missouri;
he would like to thank the people at this department for their support and hospitality.}

\thanks{The authors thank Sylvy Anscombe, Arno Fehm, Hagen Knaf and Josnei Novacoski 
for inspiring discussions and helpful feedback, and Anna Rzepka for careful 
proofreading of an earlier version.
}

\address{Department of Mathematics, University of Missouri, Columbia,
MO 65211, USA}
\email{cutkoskys@missouri.edu}

\address{Institute of Mathematics, University of Szczecin,	
ul. Wielkopolska 15, 	  	  	
70-451 Szczecin, Poland}
\email{fvk@usz.edu.pl}

\begin{abstract}\noindent
Assume that $(L,v)$ is a finite Galois extension of a valued field $(K,v)$. We give an
explicit construction of the valuation ring $\cO_L$ of $L$ as an $\cO_K$-algebra, and 
an explicit description of the module of relative K\"ahler differentials 
$\Omega_{\cO_L|\cO_K}$ when $L|K$ is a Kummer extension of prime degree or an 
Artin-Schreier extension, in terms of invariants of the valuation and 
field extension. The case when this extension has nontrivial defect was solved in a 
recent paper by the authors with Anna Rzepka. The present paper deals 
with the complementary (defectless) case. The results are known 
classically for (rank 1) discrete valuations, but our systematic 
approach to non-discrete valuations (even of rank 1) is new.

\pars
Using our results from the prime degree case, we characterize when 
$\Omega_{\cO_L|\cO_K}=0$ holds for an arbitrary finite Galois extension of valued 
fields. As an application of these results, we give a simple proof of 
a theorem of Gabber and Ramero, which characterizes when a valued field is deeply 
ramified. We further give a simple characterization of deeply ramified fields 
with residue fields of characteristic $p>0$ in terms of the K\"ahler 
differentials of Galois extensions of degree $p$.
\end{abstract}


\maketitle
%
%
\section{Introduction}
%
The main goal of this paper is to study for algebraic extensions 
of valued fields the relation between their properties and 
the vanishing of the K\"ahler differentials of the extensions of their valuation rings. 

 All of our results are for
arbitrary valuations; in particular, we have no restrictions on their rank or value 
groups. Ranks higher than $1$ appear in a natural way when local uniformization, the 
local form of resolution of singularities, is studied. Deeply ramified fields of 
infinite rank appear in model theoretic investigations of the tilting construction, as 
presented by Jahnke and Kartas in \cite{JK}.
Therefore, we do not restrict our computations to rank $1$, thereby indicating how 
Kähler differentials
can be computed in higher rank.

\pars
The notation we use is mostly standard in valuation theory or 
commutative algebra. We review notation and some main notions in Section \ref{sectnot}.

\pars
Our principal result is the following Theorem \ref{ThmI1}, which deals with extensions $L|K$
which are Kummer extensions of prime degree or Artin-Schreier extensions. In this paper we
compute the K\"ahler differentials $\Omega_{\cO_L|\cO_K}$ for such extensions when they
are \bfind{unibranched} and \bfind{defectless}, which  means that the extension of $v$ from 
$K$ to $L$ is unique and $[L:K]=(vL:vK)[Lv:Kv]$ holds. For the complementary case of such 
extensions with \bfind{nontrivial defect}, which in this special case means 
that $(vL:vK)=1=[Lv:Kv]$, see Theorems~4.5 and~4.6
in the recent paper \cite{CuKuRz} by the authors with Anna Rzepka. The description of 
these K\"ahler differentials is known classically for (rank 1) discrete valuations, but 
our systematic and detailed description is new, even for arbitrary valuations of rank 1. 
By $\Omega_{B|A}$ we denote the Kähler differentials, i.e., 
the module of relative differentials, when $A$ is a ring and $B$ is an $A$-algebra.

\begin{theorem}\label{ThmI1}
Let $(L|K,v)$ be a finite Galois extension of valued fields where $L|K$ is a Kummer
extension of prime degree or an Artin-Schreier extension. Then there is an explicit
description of $\Omega_{\cO_L|\cO_K}$ in terms of invariants of the valuation $v$ and 
field extension $L|K$. This gives a characterization of when $\Omega_{\cO_L|\cO_K}=0$.
\end{theorem}

The proof of Theorem \ref{ThmI1} is given in Section \ref{secKahExact}, after Proposition \ref{KahCalc20*}.
The analysis of the cases in Theorem \ref{ThmI1} begins with explicit constructions of 
the extensions $\cO_L|\cO_K$ of valuation rings, as a chain of simple ring extensions. 
This construction depends strongly on the type of extension. For the case of defectless
extensions it is given in Section~\ref{sectgen}; to the best of our knowledge, it is new 
and of independent interest. A result from \cite{CuKuRz}, stated in Proposition~\ref{LimProp} 
of the present paper, is then used to give the explicit description 
of $\Omega_{\cO_L|\cO_K}$ in Sections~\ref{sectseprfe} to~\ref{sectOKume}.

\pars
Annihilators of $\Omega_{\cO_L|\cO_K}$, differents $\cD_{\cO_L|\cO_K}$ and traces of the
maximal ideal $\cM_L$ of $\cO_L$ for the extensions appearing in Theorem \ref{ThmI1} have
been determined in \cite{CuKuRz} in the case of nontrivial defect. (Note that before
\cite[Theorem 1.6]{CuKuRz} we meant to write ``We denote the annihilator of an
$\cO_\cE$-module $M$ by $\ann M$''.)
The case of defectless extensions will be addressed in~\cite{KuTII}.

As an application of Theorem \ref{ThmI1}, we prove in Section~\ref{secKahExact} a 
criterion for the vanishing of the K\"ahler differentials for arbitrary finite Galois
extensions; see part 2) of Theorem~\ref{Kahler}. Finally, in Section~\ref{sectprf} all 
of these results are combined into the proof of the next theorem.

Take a valued field $(K,v)$ with valuation ring $\cO_K\,$. Choose any extension of $v$ 
to the separable-algebraic closure $K\sep$ of $K$ and denote the valuation ring of 
$K\sep$ with respect to this extension by $\cO_{K\sep}\,$. Note that $\Omega_{\cO_{K\sep}|
\cO_K}$ does not depend on the choice of the extension of $v$ since all of the possible
extensions are conjugate. Gabber and Ramero prove the following result (see 
\cite[Theorem~6.6.12 (vi)]{GR}):
\begin{theorem}              \label{GRthm}
For a valued field $(K,v)$, 
\begin{equation}                         \label{GRdefdr}
\Omega_{\cO_{K\sep}|\cO_K} \>=\> 0
\end{equation}
holds  if and only if it satisfies the following:
\sn
{\bf (DRvg)} whenever $\Gamma_1\subsetneq\Gamma_2$ are convex subgroups of the value
group $vK$, then $\Gamma_2/\Gamma_1$ is not isomorphic to $\Z$ (that is, no
archimedean component of $vK$ is discrete);
\sn
{\bf (DRvr)} if $\chara Kv=p>0$, then the homomorphism
\begin{equation}                          \label{homOpO}
\cO_{\hat K}/p\cO_{\hat K} \ni x\mapsto x^p\in \cO_{\hat K}/p\cO_{\hat K}
\end{equation}
is surjective, where $(\hat K,\hat v)$ is the completion of $(K,v)$ for the valuation 
topology and $\cO_{\hat K}$ denotes its valuation ring. 
\end{theorem}

Theorem~\ref{GRthm} and the papers \cite{Th1,Th2} of Thatte were the motivation for our 
work in the present paper and in \cite{CuKuRz}. 

For the purpose of the proof of Theorem~\ref{GRthm}, we define (as we have done in 
\cite{KuRz}) a nontrivially valued field $(K,v)$ to be a \bfind{deeply ramified field} 
if the conditions (DRvg) and (DRvr) hold.
In \cite{KuRz}, related classes of valued fields are introduced 
by weakening or strengthening condition (DRvg).

Note that by \cite[Definition~3.1]{Sch} a perfectoid field is a complete nondiscrete
rank 1 valued field of positive residue characteristic such that the Frobenius is surjective on $\cO_K/p\cO_K\,$. In rank 1, condition (DRvg) just says that the value
group is not discrete. Consequently, when using (DRvg) and (DRvr) for the definition of
deeply ramified fields, it is immediately seen that every perfectoid field is a deeply
ramified field.

\pars
The proof of Theorem \ref{GRthm} in \cite{GR} is a demonstration of the power of the
techniques of almost ring theory, and uses a large part of the theory developed in 
\cite{GR}. The proof is by reduction to the rank 1 case, where the techniques of 
almost ring theory are most applicable. 

Our alternative proof of Theorem~\ref{GRthm} in the present paper uses 
only methods from valuation theory and commutative algebra, and does not rely on 
techniques or results from almost ring theory. We hope that our proof makes this 
beautiful theorem accessible to a wider audience. Further, our proof yields the following
additional new result. A criterion for a valued field $(K,v)$ to be deeply ramified that 
only works with extensions of prime degree $p=\chara Kv$ appears to be more easily accessible
than the criterion $\Omega_{\cO_{K\sep}|\cO_K}=0$, in particular from the model theoretic 
point of view.
\begin{theorem}                              \label{GRthm+}
Let $(K,v)$ be a valued field of residue characteristic $p>0$. If $K$ has characteristic
$0$, then assume in addition that it contains all $p$-th roots of unity. Then $(K,v)$ is 
a deeply ramified field if and only if $\Omega_{\cO_L|\cO_K}=0$ for all unibranched
Galois extensions $(L|K,v)$ of prime degree $p$.
\end{theorem}

Let us mention two main ingredients of the proof. Theorem~1.10 (1) of \cite{KuRz} implies 
that if $(K,v)$ is a deeply ramified field with $\chara Kv=p>0$, then each of its Galois 
defect extensions of degree $p$ has independent defect. Hence we can infer the 
following result from \cite[Theorem~1.4]{CuKuRz}:
\begin{theorem}                            \label{MTsdrf1}
Take a deeply ramified field $(K,v)$ with $\chara Kv=p>0$; if $\chara K=0$, then 
assume that $K$ contains all $p$-th roots of unity. Then every Galois extension 
$(L|K,v)$ of degree $p$ with nontrivial defect satisfies $\Omega_{\cO_L|\cO_K}=0$.
\end{theorem}

This result will be complemented in the present paper by showing that for a deeply 
ramified field $(K,v)$, every unibranched defectless Galois extension $(L|K,v)$ of 
prime degree $p$ satisfies $\Omega_{\cO_L|\cO_K}=0$. Then Section~\ref{secKahExact}
connects our
results for Galois extensions of prime degree with $\Omega_{\cO_{K\sep}|\cO_K}$. 
There, the main approach is the study of K\"ahler differentials of towers of 
Galois extensions. In order to go upward through such towers, we make use of the 
following fact, which Gabber and Ramero deduce from Theorem~\ref{GRthm} (see 
\cite[Corollary 6.6.16 (i)]{GR}). However, as we want to prove Theorem~\ref{GRthm},
we refer the reader to Theorem~1.5 of \cite{KuRz} whose proof is done by a direct 
valuation theoretical computation not involving any K\"ahler differentials.
\begin{theorem}                            \label{MTsdrf2}
Every algebraic extension of a deeply ramified field is again a deeply ramified 
field.
\end{theorem}
It should be noted that 
Theorem~\ref{MTsdrf2} also holds for 
the roughly deeply ramified and the semitame fields that are introduced in \cite{KuRz}.

\pars
In \cite{NS}, Novacoski and Spivakovsky use the theory of key polynomials to derive a
presentation of $\Omega_{\cO_L|\cO_K}$ for finite pure extensions $(L|K,v)$ under the 
condition $vL=vK$. 
Applying this presentation to Artin-Schreier and Kummer extensions, 
they derive results similar to our results presented in \cite{CuKuRz} and in this paper.
Recently they also dealt with the case of $vL\ne vK$ by a different approach, not based 
on the use of key polynomials. See also \cite{NN1, NN2, No}.

\parm
To conclude this introduction, let us give some interesting examples. Let $\zeta_p$
denote a primitive $p$-th root of unity.
\begin{example}
Choose a prime $p>2$. The field $K=\Q_p(\zeta_p,p^{1/p^n}\mid n\in \N)$, equipped with the
unique extension of the $p$-adic valuation of $\Q_p$, is known to be a deeply ramified
field. The Kummer extension $(K(\sqrt{p})|K,v_p)$ is tamely ramified, as $(v_pK(\sqrt{p}):
v_pK)=2\ne p$.
By an application of Theorem~\ref{OKume} below, $\Omega_{\cO_{K(\sqrt{p})}|\cO_K}=0$.
The fact that this holds in spite of the ramification is due to the value
group $v_pK$ being dense, as it is $p$-divisible.

\pars
Analoguously, we can consider the field $K=\F_p((t))(t^{1/p^n}\mid n\in \N)$, 
equipped with the unique extension of the $t$-adic valuation of $\F_p((t))$.
This field is a deeply ramified field since it is perfect of positive characteristic. 
Again, the extension $(K(\sqrt{t})|K,v_t)$ is tamely ramified as $(v_tK(\sqrt{t}):
v_tK)=2\ne p$, and $v_t K$ is dense. By Theorem~\ref{OKume} below, 
$\Omega_{\cO_{K(\sqrt{t})}|\cO_K}=0$.
\end{example}

Finally, here is an example of a Kummer extension $(L|K,v)$ with 
wild ramification and $\Omega_{\cO_L|\cO_K}=0$.
\begin{example}
Take a prime $p>2$ and set $K=\Q(\zeta_p)(t^{1/2^n}\mid n\in \N)$. Let $v_p$ denote the
$p$-adic valuation on $\Q(\zeta_p)$
and $v_t$ the $t$-adic valuation on $K$. Now consider the valuation $v:=v_t\circ v_p$ on
$K$, where ``$v_t\circ v_p$'' denotes the valuation associated with the composition of
the $t$-adic place on $K$ and the $p$-adic place on $\Q(\zeta_p)$. Set $L=K(t^{1/p})$
and extend $v$ to $L$. Then
$(L|K,v)$ is a Kummer extension of degree $p$ with ramification index $p=\chara Kv$.
Nevertheless, Theorem~\ref{OKume} shows that $\Omega_{\cO_L|\cO_K}=0$.
\end{example}

\bn
%
%
\section{Preliminaries}
%
%
%
%
\subsection{Notation}                        \label{sectnot}  
\mbox{ }\sn 
By $(L|K,v)$ we denote a field extension $L|K$ where $v$ is a
valuation on $L$ and $K$ is endowed with the restriction of $v$. The valuation
ring of $v$ on $L$ will be denoted by $\cO_L\,$, and that on $K$ by $\cO_K\,$. 
Similarly, $\cM_L$ and $\cM_K$ denote the unique maximal ideals of $\cO_L$ and 
$\cO_K$. The value group of the valued field $(L,v)$ will be denoted by $vL$, and
its residue field by $Lv$. The value of an element $a$ will be denoted by $va$, and 
its residue by $av$. In order to simplify notation by reducing the use of brackets, our
convention will be that $v\ldots$ denotes the value of the term following ``$v$'', and 
$\ldots v$ denotes the residue of the term preceding ``$v$''; for example, $vxy=v(xy)$
and $xyv=(xy)v$. A {\bf final segment} of $vL$ is a subset $S\subseteq vL$ such that
$\gamma\leq\delta$ with $\gamma,\delta\in vL$ and $\gamma\in S$ implies that $\delta\in
S$.

The \bfind{rank} of a valued field $(K,v)$ is the order type of the
chain of proper convex subgroups of its value group $vK$. We say that $(L|K,v)$ is 
unibranched if the extension of $v$ from $K$ to $L$ is unique.


\sn
%
%
\subsection{Convex subgroups and archimedean components}     \label{sectcsac}
\mbox{ }\sn
Take an ordered abelian group $\Gamma$. Two elements $\alpha,\beta\in \Gamma$ are 
\bfind{archime\-dean equivalent} if there is some $n\in\N$ such that $n|\alpha|\geq
|\beta|$ and $n|\beta|\geq |\alpha|$, where $|\alpha|:=\max\{\alpha,-\alpha\}$. Note 
that if $0<\alpha<\beta<n\alpha$ for some $n\in\N$, then $\alpha$, $\beta$ and $n\alpha$ 
are (mutually) archimedean equivalent. If every two nonzero elements of $\Gamma$
are archimedean equivalent, then we say that $\Gamma$ is \bfind{archimedean ordered}.
This holds if and only if $\Gamma$ admits an order preserving embedding in the ordered
additive group of the real numbers.

We call $\Gamma$ \bfind{discretely ordered} if every element in $\Gamma$ has
an immediate successor; this holds if and only if $\Gamma$ contains a smallest positive 
element. In contrast, $\Gamma$ is called \bfind{dense} if $\Gamma\ne\{0\}$ and
for every two elements $\alpha<
\gamma$ in $\Gamma$ there is $\beta\in\Gamma$ such that $\alpha<\beta<\gamma$. If
$\Gamma$ is archimedean ordered and dense, then for every $i\in\N$ there is even some 
$\beta_i\in \Gamma$ such that $\alpha<i\beta_i<\gamma$; this can be easily proven via an 
embedding of $\Gamma$ in the real numbers. Every ordered abelian group is discrete if
and only if it is not dense.

For $\gamma\in\Gamma$, we define $\cC_\Gamma(\gamma)$ to be the smallest convex 
subgroup of $\Gamma$ containing $\gamma$, and for $\gamma\ne 0$, $\cC_\Gamma^+(\gamma)$ 
to be the largest convex subgroup of $\Gamma$ not containing $\gamma$. Note that 
$\cC_\Gamma(0)=\{0\}$. The convex subgroups of $\Gamma$ form a chain under inclusion, 
and the union and intersection of any collection of convex subgroups are 
again convex subgroups; this guarantees the existence of $\cC_\Gamma(\gamma)$ and
$\cC_\Gamma^+(\gamma)$.

We have that $\cC_\Gamma^+(\gamma)\subsetneq\cC_\Gamma(\gamma)$ and that 
$\cC_\Gamma^+(\gamma)$ and $\cC_\Gamma(\gamma)$ are consecutive, that is, there is 
no convex subgroup of $\Gamma$ lying properly between them. As a consequence, 
\[
\cA_\Gamma (\gamma) \>:=\> \cC_\Gamma(\gamma)/\cC_\Gamma^+(\gamma)
\]
for $\gamma\ne 0$ is an archimedean ordered group; we call it the \bfind{archimedean
component of $\Gamma$ associated with $\gamma$}. Two elements $\alpha,\beta\in\Gamma$
are archimedean equivalent if and only if 
\[
\cC_\Gamma (\alpha)\>=\>\cC_\Gamma (\beta)\>,
\]
and then it follows that $\cA_\Gamma (\alpha)=\cA_\Gamma (\beta)$. In particular, 
$\cC_\Gamma (\alpha)=\cC_\Gamma (n\alpha)$ and $\cA_\Gamma (\alpha)=\cA_\Gamma (n\alpha)$
for all $\alpha\in\Gamma$ and all $n\in\Z\setminus\{0\}$.

Assume now that $\Gamma$ is an ordered abelian group containing a subgroup $\Delta
\ne\{0\}$. We say that \bfind{$\Delta$ is dense in $\Gamma$} if for every two elements 
$\alpha<\gamma$ in $\Gamma$ there is $\beta\in\Delta$ such that $\alpha<\beta<\gamma$;
this implies that $\Gamma$ and $\Delta$ are dense. If $\Gamma$ is archimedean ordered,
then so is $\Delta$, and $\Delta$ is dense in $\Gamma$ if and only if it is dense.

For every $\gamma\in\Gamma$, $\cC_\Gamma(\gamma)\cap\Delta$ and $\cC_\Gamma^+(\gamma)
\cap\Delta$ are convex subgroups of $\Delta$; the quotient $\cC_\Gamma(\gamma)\cap
\Delta\,/\,\cC_\Gamma^+(\gamma)\cap\Delta$ is either trivial or archimedean ordered. 
If $\gamma$ is archimedean equivalent to $\delta\in\Delta$, then this quotient is 
equal to $\cA_\Delta(\delta)$.

For each $\delta\in\Delta$ the function given by
\[
\cA_\Delta (\delta)\ni\alpha+\cC_\Delta^+(\delta)\,\mapsto\, \alpha+\cC_\Gamma^+(\delta)
\in\cA_\Gamma (\delta)
\]
is an injective order preserving homomorphism. This follows from the fact that the kernel
of the homomorphism $\cC_\Delta (\delta)\ni\alpha\mapsto\alpha+\cC_\Gamma^+(\delta)\in
\cA_\Gamma (\delta)$ is the convex subgroup $\cC_\Delta^+ (\delta)=\cC_\Gamma^+(\delta)
\cap\Delta$. In abuse of notation, we write $\cA_\Delta (\delta)=\cA_\Gamma (\delta)$ if
this homomorphism is surjective.

\sn
%
%
%
\subsection{Artin-Schreier and Kummer extensions}        \label{sectASKext}
\mbox{ }\sn
We say that a valued field $(K,v)$ has \bfind{equal characteristic} if $\chara K=
\chara Kv$, and \bfind{mixed characteristic} if $\chara K=0$ and $\chara Kv>0$.
Every Galois extension of degree $p$ of a field $K$ of characteristic $p>0$ is an
\bfind{Artin-Schreier extension}, that is, generated by an \bfind{Artin-Schreier
generator} $\vartheta$ which is the root of an \bfind{Artin-Schreier polynomial} 
$X^p-X-b$ with $b\in K$. For every $c\in K$, also $\vartheta-c$ is an Artin-Schreier
generator as its minimal polynomial is $X^p-X-b+c^p-c$. Every Galois extension 
of prime degree $q$ of a field $K$ of characteristic different from $q$ which contains 
all $q$-th roots of unity is a \bfind{Kummer
extension}, that is, generated by a \bfind{Kummer generator} $\eta$ which satisfies 
$\eta^q\in K$. For these facts, see \cite[Chapter VI, \S6]{[L]}.

A \bfind{1-unit} in a valued field $(K,v)$ is an element of the form $u=1+b$ with 
$b\in\cM_K\,$; in other words, $u$ is a unit in $\cO_K$ with residue 1. We note that 
if $u$ is a $1$-unit, then also $u^{-1}$ is a $1$-unit, and if $v(u-c)>vu=0$ for some 
$c\in K$, then also $c$ is a 1-unit. Conversely, if $u$ and $c$ are 1-units, then $v(u-c)>0$. 

\begin{remark}                           \label{rem1-un}
Take a Kummer extension $(L|K,v)$ of degree $p$ with 
any Kummer generator $\eta$. Assume that $v\eta\in vK$, so that there is $c_1\in K$ 
such that $vc_1=-v\eta$, whence $vc_1\eta =0$. Assume further that $c_1\eta v\in Kv$, 
so that there is $c_2\in K$ such that $c_2v=(c_1\eta v)^{-1}$. Then $vc_2c_1\eta=0$ 
and $c_2c_1\eta v=1$. Furthermore, $K(c_2c_1\eta)=K(\eta)$ and $(c_2c_1\eta)^p=
c_2^pc_1^p\eta^p\in K$. Hence $c_2c_1\eta$ is a Kummer generator of $(L|K,v)$
and a $1$-unit. Therefore $v(c_2c_1\eta-1)>0$, whence $v(\eta-(c_2c_1)^{-1})>
v(c_2c_1)^{-1}=v\eta$. Consequently, for $c:=(c_2c_1)^{-1}\in K$ we have
$v(\eta-c)>v\eta$.
\end{remark}

\pars
We will need the following facts. 
If $(L|K,v)$ is a unibranched defectless extension of prime degree $p$, then either 
$\rme(L|K,v)=1$ and $\rmf(L|K,v)=p$, or $\rmf(L|K,v)=1$ and $\rme(L|K,v)=p$.
For $q\in\N$ let $\zeta_q$ denote a primitive $q$-th root of unity. We note that if 
$L|K$ is a Kummer extension of degree $q$, then $K$ contains all $q$-th roots of unity.
For a proof of the next well known results, see \cite[Lemma 2.5]{CuKuRz}.
\begin{lemma}
Take  $q\in\N$ and a valued field $(K,v)$ containing $\zeta_q\,$. Then 
\begin{equation}                       \label{prod1-z}
\prod_{i=1}^{q-1} (1-\zeta_q^i)\>=\> q \>.
\end{equation}
If in addition $q$ is prime, then
\begin{equation}                       \label{vz-1}
v(\zeta_q-1)\>=\>\frac{vq}{q-1} \>.  
\end{equation}
\end{lemma}

\begin{lemma}                           \label{v(eta-1)}   
Take a unibranched Kummer extension $(L|K,v)$ of prime degree $q$ with Kummer 
generator $\eta$. Then for all $c\in K$,
\begin{equation}                           \label{eqv(eta-c)}
v(\eta-c)\>\leq\>v\eta(\zeta_q-1)\>=\>v\eta+\frac{vq}{q-1} \>.  
\end{equation}

Assume in addition that $\rmf (L|K,v)=q=\chara Kv$ and $c,\tilde{c}\in K$ are such that 
$v\tilde{c}(\eta-c)=0$ and $\tilde{c}(\eta-c)v$ generates the residue field extension
$Lv|Kv$. Then $Lv|Kv$ is inseparable if and only if $v(\eta-c)<v\eta(\zeta_q-1)$, and it
is separable if and only if $v(\eta-c)=v\eta(\zeta_q-1)$.
\end{lemma}
\begin{proof}
Take $c\in K$ and $\sigma\in\Gal L|K$ such that $\sigma\eta=\zeta_q\eta$. Then 
\begin{equation}                            \label{eta-c}
\eta-c -\sigma(\eta-c)\>=\>\eta-\sigma\eta\>=\>\eta(1-\zeta_q)\>.
\end{equation}
Hence if $v(\eta-c)>v\eta(1-\zeta_q)$, then
\[
v\sigma(\eta-c)\>=\>v(\eta-c -\eta(1-\zeta_q))\>=\>\min\{v(\eta-c),v\eta(1-\zeta_q)\}
\>=\> v\eta(1-\zeta_q)\><\>v(\eta-c)\>,
\]
which shows that $v\sigma\ne v$, i.e., the extension is not unibranched. This
contradiction proves the first assertion.

Now assume the situation as in the second part of the lemma. Since $L|K$ is a 
Galois extension, $Lv|Kv$ is a normal extension, with its automorphisms
induced by those of $L|K$. Take $\sigma$ to be a generator of $\Gal L|K$. Via the 
residue map, its action on $\cO_L^\times$ induces a generator $\bar\sigma$ of the
automorphism group of $Lv|Kv$. From (\ref{eta-c}) we infer that
\[
\tilde{c}(\eta-c)-\sigma \tilde{c}(\eta-c)\>=\>
\tilde{c}\eta(1-\zeta_q)\>.
\]
It follows that $\bar\sigma$ is the identity, i.e, $Lv|Kv$ is inseparable, if and 
only if $v\tilde{c}\eta(1-\zeta_q)>0$. This is equivalent to $v(\eta-c)=-v\tilde{c}<
v\eta(1-\zeta_q)$. Since $v(\eta-c)>v\eta(1-\zeta_q)$ is impossible according to
(\ref{eqv(eta-c)}), we can conclude that the residue field extension is separable 
if and only if $v(\eta-c)=v\eta(1-\zeta_q)$. 
\end{proof}

\begin{proposition}                   \label{trKext}
Take a Kummer extension $(L|K,v)$ of prime degree $q\ne\chara Kv$. 
\sn
1) If $\,\rmf (L|K,v)=q$, then there is a Kummer generator $\eta\in
\cO_L^\times$ such that $\eta v$ is a Kummer generator of $Lv|Kv$. 
\sn
2) If $\,\rme (L|K,v)=q$, then there is a Kummer generator $\eta\in L$ 
such that $v\eta$ generates the value group extension, that is, $vL=vK+\Z v\eta$.
\end{proposition}
\begin{proof}
Since $q\ne\chara Kv$, we have $vq=0$ and thus $v(1-\zeta_q)=0$.
\n
1): Take a Kummer generator $\eta$. Since $\rmf (L|K,v)=q$, we have that $vL=vK$.
Therefore, as shown in Remark~\ref{rem1-un}, we can assume that $v\eta=0$. The
reduction of the minimal polynomial of $\eta$ over $K$ to the residue field is $X^q-
\eta^q v$ with $\eta^q v\ne 0$. Suppose that this polynomial has a root in $Kv$.  
Since $\Gal Lv|Kv$ is cyclic (generated by the reduction of a 
generator of $\Gal L|K$), it follows that $X^q-\eta^q v$ splits. Hence its root 
$\eta v$ lies in $Kv$ and there is $c\in K$ such that $cv=\eta v$. It follows that 
$v(\eta-c)>0=v\eta(1-\zeta_q)$, so by Lemma~\ref{v(eta-1)}, $(L|K,v)$ is
not unibranched. As this contradicts our assumption, $X^q-\eta^q v$ must be irreducible
(cf.\ \cite{stackex}), which means that $\eta v$ generates the extension $Lv|Kv$. Since 
$\eta^q\in K$, we have that $(\eta v)^q\in Kv$, i.e., $\eta v$ is a Kummer generator of
$Lv|Kv$.

\sn
2) Take a Kummer generator $\eta$. We will show that $v\eta\notin vK$; as $q$ is prime,
it then follows that $vL=vK+\Z v\eta$. Suppose that $v\eta\in vK$. Since $\rme (L|K,v)
=q=[L:K]$, we have that $Lv=Kv$. Thus as shown in Remark~\ref{rem1-un}, there is some
$c\in K$ such that $v(\eta-c)>v\eta=v\eta(1-\zeta_q)$. As in the proof of part 1), 
this leads to a contradiction. Hence $v\eta\notin vK$, as asserted.
\end{proof}

For the next lemma, see 
\cite[Lemma 2.1]{KuVl} and the proof of \cite[Theorem 2.19]{Ku30}.
\begin{lemma}                                         \label{dist_defectless}
If $(L|K,v)$ is a finite unibranched defectless extension, then 
for every element $x\in L$ the set 
\[
v(x-K)\>:=\>\{v(x-c)\mid c\in K\}
\]
admits a maximal element. If $c\in K$ is 
such that $v(x-c)$ is maximal, then $v(x-c)\notin vK$ or otherwise, for every $\tilde{c}
\in K$ such that $v\tilde{c}(x-c)=0$ we have $\tilde{c}(x-c)v\notin Kv$.
\end{lemma}

Using this lemma, we prove:
\begin{proposition}                               \label{ASKgen}
1) Take a valued field $(K,v)$ of equal positive characteristic $p$ and a unibranched
defectless Artin-Schreier extension $(L|K,v)$. 

If $\rmf(L|K,v)=p$, then the extension has an Artin-Schreier generator $\vartheta$ of 
value $v\vartheta\leq 0$ such that $Lv=Kv(\tilde{c}\vartheta v)$ for every $\tilde{c}
\in K$ with $v\tilde{c}\vartheta=0$; the extension $Lv|Kv$ is separable if and only if
$v\vartheta=0$. 

If $\rme(L|K,v)=p$, then the extension has an Artin-Schreier generator $\vartheta$ such 
that $vL=vK+\Z v\vartheta$. Every such $\vartheta$ satisfies $v\vartheta<0$.

\sn
2) Take a valued field $(K,v)$ of mixed characteristic and a unibranched defectless 
Kummer extension $(L|K,v)$ of degree $p=\chara Kv$. Then the extension has a Kummer 
generator $\eta$ such that: 
\sn
a) if $\,\rmf(L|K,v)=p$, then either $\eta v$ generates the residue field extension, 
in which case it is inseparable, or $\eta$ is a $1$-unit and for some $\tilde{c}\in K$, 
$\tilde{c}(\eta-1)v$ generates the residue field extension;
\sn
b) if $\,\rme(L|K,v)=p$, then either $v\eta$ generates the value group extension, or 
$\eta$ is a $1$-unit and $v(\eta-1)$ generates the value group extension.
\end{proposition}
\begin{proof}
1): Take any Artin-Schreier generator $y$ of $(L|K,v)$. Then by 
Lemma~\ref{dist_defectless} there is $c\in K$ such that either $v(y-c)\notin vK$, or
for every $\tilde{c}\in K$ such that $v\tilde{c}(x-c)=0$ we have $\tilde{c}(y-c)v\notin 
Kv$. Since $p$ is prime, 
in the first case it follows that $\rme(L|K,v)=p$ and that $v(y-c)$ generates the value
group extension. In the second case it follows that $\rmf(L|K,v)=p$ and that $\tilde{c}
(y-c)v$ generates the residue field extension. In both cases, $\vartheta=y-c$ is an 
Artin-Schreier generator. Let $\vartheta^p-\vartheta=b\in K$.

Assume that $\rmf(L|K,v)=p$. If $v\vartheta<0$, then $v(\vartheta^p-b)=v\vartheta>
pv\vartheta=v\vartheta^p$, whence $v((\tilde{c}\vartheta)^p-\tilde{c}^pb)=
v\tilde{c}^p\vartheta>v(\tilde{c}\vartheta)^p$ for $\tilde{c}\in K$ with $v\tilde{c}\vartheta=0$ and therefore, $(\tilde{c}\vartheta)^p v
=\tilde{c}^pb v\in Kv$. In this case, the residue field 
extension is inseparable. Now assume that $v\vartheta\geq 0$ and hence also $vb\geq 0$.
The reduction of $X^p-X-b$ to $Kv[X]$ is a separable polynomial, so $Lv|Kv$ is 
separable. The polynomial $X^p-X-bv$ cannot have a zero in $Kv$, since otherwise 
the $p$ distinct roots of this polynomial give rise to $p$ distinct extensions of $v$
from $K$ to $L$, contradicting our assumption that $(L|K,v)$ is unibranched. 
Consequently, $bv\ne 0$, whence $vb=0$ and $v\vartheta= 0$.

Assume that $\rme(L|K,v)=p$. If 
$v\vartheta\geq 0$, then $vb\geq 0$ and $\vartheta v$ is a root of $X^p-X-bv$. If this 
polynomial does not have a zero in $Kv$, then $\vartheta v$ generates a nontrivial residue 
field extension, contradicting our assumption that $\rme(L|K,v)=p$. If the polynomial has 
a zero in $Kv$, then similarly as before one deduces that
$(L|K,v)$ is not unibranched, contradiction. Hence $v\vartheta<0$.

\sn
2): Take any Kummer generator $y$ of $(L|K,v)$. If there is a Kummer generator $\eta$
such that $v\eta\notin vK$, then it follows as before that $\rme(L|K,v)=p$ and that 
$v\eta$ generates the value group extension. Now assume that there is no such $\eta$.

If there is a Kummer generator $y$ and some $\tilde{c}\in K$ such that $v\tilde{c}y=0$ 
and $\tilde{c}yv\notin Kv$, then it follows as before that $\rmf(L|K,v)=p$ and that 
$\tilde{c}yv$ generates the residue field extension. We set $\eta=\tilde{c}y$ and 
observe that also $\eta$ is a Kummer generator.
Since $(\eta v)^p\in Kv$, $Lv|Kv$ is purely inseparable in this case.

Now assume that the above cases do not appear, and choose an arbitrary Kummer generator 
$y$ of $(L|K,v)$. Consequently, we have that $vy\in vK$ and $\tilde{c}yv\in Kv$ for all 
$\tilde{c}\in K$ with $v\tilde{c}y=0$. Then as described in Remark~\ref{rem1-un}, there 
are $c_1,c_2\in K$ such that $c_2c_1y$ is a Kummer generator of $(L|K,v)$ which is a 
$1$-unit. We replace $y$ by $c_2c_1y$.

By Lemma~\ref{dist_defectless} there is $c\in K$ such that $v(y-c)$ is maximal in 
$v(y-K)$ and either $v(y-c)\notin vK$ or there is some $\tilde{c}\in K$ 
such that $v\tilde{c}(y-c)=0$ and $\tilde{c}(y-c)v\notin Kv$. Since $y$ is a $1$-unit,
we know that $v(y-1)>0$, hence also $v(y-c)>0=vy$, showing that also $c$ is a $1$-unit.
Then $\eta:=c^{-1}y$ is again a Kummer generator of $(L|K,v)$ which is a $1$-unit. Since
$vc=0$, we know that $v(\eta-1)=vc(\eta-1)=v(y-c)$. Hence if $v(y-c)\notin vK$, then 
$v(\eta-1)$ generates the value group extension.

Now assume that there is $\tilde{c}\in K$ such that $v\tilde{c}(y-c)=0$ and 
$\tilde{c}(y-c)v\notin Kv$. Since $c$ is a $1$-unit, it follows that $v\tilde{c}(\eta-1)
=v\tilde{c}c(\eta-1)=v\tilde{c}(y-c)=0$ and $\tilde{c}(\eta-1)v=\tilde{c}c
(\eta-1)v=\tilde{c}(y-c)v$. We find that $\tilde{c}(\eta-1)v$ generates the residue field
extension.
\end{proof}

\sn
%
%
%
\subsection{Ramification ideals}        \label{sectri}
\mbox{ }\sn
Take a unibranched Galois extension $\cE=(L|K,v)$ and let $G=\Gal L|K$ denote its 
Galois group. An $\cO_L$-ideal 
\begin{equation}                     \label{genI}
\left(\frac{\sigma b -b}{b}\>\left|\;\> \sigma\in H\,,\>b\in L^\times\right.
\right)\>,
\end{equation}
where $H$ is a nontrivial subgroup of $G$, is called a \bfind{ramification ideal} of 
$\cE$. Hence if $\cE$ is of prime degree, then it has a unique ramification ideal, which we 
denote by $I_\cE\,$. For further background on ramification ideals, see
\cite{CuKuRz,KuRz,KuTI}. In \cite{KuTI}, the following is shown:
\begin{proposition}                              \label{PropASKgen}
Take a unibranched defectless Galois extension $\cE=(L|K,v)$ of prime degree $q$. 
\sn
1) Let $(L|K,v)$ be an Artin-Schreier extension and $\vartheta$ an Artin-Schreier
generator as in part 1) of Proposition~\ref{ASKgen}. Then 
\begin{equation}
I_{\cE}\>=\> \left(\frac{1}{\vartheta} \right)\>.
\end{equation}
We have $I_{\cE}=\cO_L$ if and only if $v\vartheta=0$, and this holds if and only if 
$Lv|Kv$ is separable of degree $q$.
\sn
2) Let $(L|K,v)$ be a Kummer extension. Then there are two cases:
\sn
a) Let $\eta$ be a Kummer generator as in part 2)a) of Proposition~\ref{ASKgen}. Then 
\begin{equation}                       
I_{\cE}\>=\> (\zeta_q-1)\>.
\end{equation}
\sn
b) Let $\eta$ be a Kummer generator as in part 2) b) of Proposition~\ref{ASKgen}. Then 
\begin{equation}                   \label{Icase2b}
I_{\cE}\>=\> \left(\frac{\zeta_q-1}{\eta-1}\right)\>.
\end{equation}
We have $I_{\cE}=\cO_L$ if and only if $v(\eta-1)=v(\zeta_q-1)$, and this holds if 
and only if $Lv|Kv$ is separable of degree $q$.
\end{proposition}

Also the ramification ideals of Artin-Schreier defect extensions and Kummer defect 
extensions of prime degree are computed in \cite{KuTI}.

\bn
%
%
\section{Generation of extensions of valuation rings} 
In this section we will assume that $\cE=(L|K,v)$ is a finite unibranched defectless
extension and develop the groundwork needed for the computation of the
K\"ahler differential of $\cE$
in Sections~\ref{sectASextf=p} to~\ref{sectOKume}.

%
%
\subsection{Generating the $\cO_K$-algebra $\cO_L$}     \label{sectgen}
\mbox{ }\sn
In order to use Proposition~\ref{LimProp} below to compute $\Omega_{\cO_L|\cO_K}$, we need 
to present $\cO_L$ as a union over a chain of simple ring extensions of 
$\cO_K\,$. We consider finite extensions $\cE=(L|K,v)$ of degree $q$ that satisfy
\[
[L:K]\>=\> [Lv:Kv]\quad \mbox{ or }\quad [L:K]\>=\> (vL:vK)\>.
\]
Such extensions are unibranched and defectless. We distinguish the following two cases:
\sn
{\bf Case (DL1)}: $[L:K]=[Lv:Kv]$.
In this case, we can choose elements $a_1,\ldots,a_q\in\cO_L^\times$ such that
$a_1v,\ldots,a_q v$ form a basis of $Lv|Kv$. Then $a_1,\ldots,a_q$ form a valuation basis
of $(L|K,v)$, which by definition means that every element of $z\in L$ can be written as
\begin{equation}                           \label{valbas}
z\,=\,c_1a_1+\ldots +c_qa_q\quad\mbox{ with }\quad vz\,=\,\min_i vc_ia_i\>,
\end{equation}
and we have that $vc_ia_i=vc_i\,$.
Consequently, $z\in\cO_L$ if and only if $c_1,\ldots,c_q\in\cO_K\,$. This shows that
$\cO_L$ is a free $\cO_K$-module with basis $a_1,\ldots,a_q\,$.

In the case where $Lv|Kv$ is simple, that is, there is $\xi\in Lv$ such that $Lv=
Kv(\xi)$, we can choose $x\in L$ such that $xv=\xi$; then $1,x,\ldots,x^{q-1}$ form a
valuation basis of $(L|K,v)$. In this special case (which by the Primitive Element
Theorem always appears when $Lv|Kv$ is separable), we have
\begin{equation}                      \label{(DL2a)}
\cO_L\>=\>\cO_K[x]\>.
\end{equation}

\sn
{\bf Case (DL2)}: $[L:K]=(vL:vK)$. We assume in addition that $q$ is a prime. In this
case we define $H_\cE$ to be the largest convex subgroup of $vL$ which is also
a convex subgroup of $vK$; it exists since unions over arbitrary collections of convex
subgroups are again convex subgroups. The subgroup $H_\cE$ defined here has important
similarities with the convex subgroup $H_\cE$ defined in the defect case in \cite{CuKuRz}.
We will discuss them in detail in \cite{KuTII}. In case (DL1) we set $H_\cE:=\{0\}$.

Now we diivide (DL2) into three mutually exclusive cases:
\sn
(DL2a): there is no smallest convex subgroup of $vL$ that properly contains $H_\cE\,$;
\sn
(DL2b): there is a smallest convex subgroup $\tilde H_\cE$ of $vL$ that properly contains
$H_\cE\,$, and the archimedean quotient $\tilde H_\cE/H_\cE$ is dense;
\sn
(DL2c): there is a smallest convex subgroup $\tilde H_\cE$ of $vL$ that properly contains
$H_\cE\,$, and the archimedean quotient $\tilde H_\cE/H_\cE$ is discrete.
\sn
We will freely use the facts outlined in Section~\ref{sectcsac}.

\pars
Pick any
$x\in L$ such that $vx\notin vK$. Then $vL=vK+\Z vx$ with $qvx\in vK$, and we have that
$1,x,\ldots,x^{q-1}$ form a valuation basis of $(L|K,v)$. This means that every element
of $L$ can be written as a $K$-linear combination of these elements
and for every choice of $c_0,\ldots,c_{q-1}\in K$,
\[
v\sum_{i=0}^{q-1} c_i x^i\>=\> \min_i vc_ix^i\>.
\]
Again, the sum is an element of $\cO_L$ if and only if all summands $c_ix^i$ are, but
the latter does not necessarily imply that $c_i\in\cO_K\,$. We set
\[
A_x\>:=\>\{c_ix^i\mid c_i\in K^\times\mbox{ and } 1\leq i<q\mbox{ such that }
vc_ix^i>0\}
\]
and
\[
vA_x\>:=\>\{va\mid a\in A_x\}\>.
\]
(Note that $vc_ix^i=0$ is impossible for $1\leq i<q$.) We obtain that
\begin{equation}              \label{Axgen}
\cO_L\>=\>\cO_K[A_x]\>.
\end{equation}
However, we wish to derive a much more useful representation of $\cO_L\,$. Our goal is 
to find an element $x$ as above such that
\begin{equation}                      \label{(DL2b)}
\cO_L\>=\> \bigcup_{c\in K\mbox{\tiny\ with } vcx>0} \cO_K[cx] \>.
\end{equation}
If $c,c'\in K$ with $vc\geq vc'\,$, then $cx=\frac{c}{c'}c' x\in
\cO_K[c' x]$, hence $\cO_K[c x] \subseteq \cO_K[c' x]$. So the right hand side 
is an increasing union of rings and thus is itself a ring. For (\ref{(DL2b)}) to hold,
it suffices that
\begin{equation}                      \label{(DL2b)A}
A_x\,\subseteq\, \bigcup_{c\in K\mbox{\tiny\ with } vcx>0} \cO_K[cx] \>.
\end{equation}
This in turn will hold if
\begin{equation}                      \label{Acx}
\left\{\begin{array}{l}
\mbox{for every element $c_ix^i\in A_x$ there is $c\in K$ with $cx\in A_x$}\\
\mbox{such that $c_ix^i\in (cx)^i\cO_K\,$.}
\end{array}\right.
\end{equation}

\begin{lemma}                            \label{HEAx}
The convex subgroup $H_\cE$ of $vL$ is the largest that has empty intersection with
$vA_x\,$.
\end{lemma}
\begin{proof}
From (\ref{Axgen}) it follows that the positve values in $vL\setminus vK$ all lie in
the smallest final segment of $vL$ generated by $vA_x\,$. On the other hand, from the
definition of $H_\cE$ it follows that it is the largest convex subgroup of $vL$ that does
not contain elements of $vL\setminus vK$. This proves our assertion.
\end{proof}

\pars
As a preparation for what follows, let us prove two useful facts.
\sn
(F1) \ For each $c_mx^m\in A_x\,$, there is $c\in K$ such that $cx\in
A_x$ and $\cC_{vL}(vc_mx^m)=\cC_{vL}(vcx)$.
\begin{proof}
As $q$ is prime, there is $k\in\N$ such that $mk=1+rq$ for some $r\in\Z$. Taking
$c:=c_m^k b^r\in K$ where $b\in K$ with $vb=qvx\,$, we obtain $vcx=
v(c_mx^m)^k>0$ and $\cC_{vL}(vcx)=\cC_{vL}(kv(c_mx^m))=\cC_{vL}(vc_mx^m)$.
\end{proof}

\sn
(F2) \ If $c_ix^i\,,\,\tilde{c}x\in A$ and $vc_ix^i\notin \cC_{vL}
(v\tilde{c}x)$, then $c_ix^i\in (\tilde{c}x)^i\cO_K\subseteq\cO_K[\tilde{c}x]$.
\begin{proof}
Since $v\tilde{c}^ix^i=iv\tilde{c}x\in \cC_{vL}(v\tilde{c}x)$ and $vc_ix^i\notin
\cC_{vL}(v\tilde{c}x)$, we have $v\tilde{c}^i x^i<vc_i x_0^i\,$. Thus $v\tilde{c}^i
<vc_i$ and therefore, $c_i x^i\in (\tilde{c}x)^i\cO_K$.
\end{proof}

\parm
Inspired by case (DL1) we ask whether (\ref{(DL2b)A}) will hold with $x=x_0$ for any
$x_0\in L$ such that $vx_0\notin vK$. We choose such an $x_0$ and set $A_0:=A_{x_0}\,$.
It can be shown that the element $x$ we are looking for cannot always be chosen to be
equal to $x_0\,$. However, we will show that in cases (DL2a) and (DL2b) it can.

\sn
(DL2a): Take any $c_ix_0^i \in A_0\,$. By assumption, $\cC_{vL}^+(vc_ix_0^i)$ properly
contains $H_\cE$. By Lemma~\ref{HEAx}, this means that $\cC_{vL}^+(vc_ix_0^i)\cap vA_0
\neq \emptyset$, so take some $c_mx_0^m\in A_0$ such that
$vc_mx_0^m\in \cC_{vL}^+(vc_ix_0^i)\cap vA_0\,$. By (F1), there is $\tilde{c}\in K$
such that $\tilde{c}x_0\in A_0$ and $\cC_{vL}(v\tilde{c}x_0)=\cC_{vL}(vc_mx_0^m)\subseteq
\cC_{vL}^+(vc_ix_0^i)$. Hence $vc_ix_0^i\notin \cC_{vL}(v\tilde{c}x_0)$ and by (F2),
$c_ix_0^i\in (\tilde{c}x_0)^i\cO_K\subseteq \cO_K[\tilde{c}x_0]$. Hence in this case,
(\ref{Acx}) and thus also (\ref{(DL2b)A}) and (\ref{(DL2b)}) hold for $x=x_0\,$.

\mn
(DL2b): By Lemma~\ref{HEAx}, $\tilde H_\cE$ is the smallest convex subgroup of $vL$ that
contains some element of $vA_0\,$, say $vc_mx_0^m$. The archimedean component $\cA_{vL}
(vc_mx_0^m)$ is equal to $\tilde H_\cE/H_\cE$, which is dense. The archimedean component
$\cA_{vK}(vqc_mx_0^m)$ is equal to $(\tilde H_\cE\cap vK)/H_\cE$. Since $(vL:vK)$ is
finite, so is this quotient. This shows that also $\cA_{vK}(vqc_mx_0^m)$ is dense, so it
is dense in $\cA_{vL}(vc_mx_0^m)$. We have $\cC_{vL}^+(c_mx_0^m)\cap vA_0=H_\cE\cap
vA_0=\emptyset$. From (F1) we know that there is $\tilde{c}\in K$ such that $\tilde{c}
x_0\in A$ and $\cC_{vL}(v\tilde{c}x_0)=\cC_{vL}(vc_mx_0^m)$.

Take any element $c_ix_0^i\in A_0\,$. If $vc_ix_0^i\notin \cC_{vL}(v\tilde{c}x_0)$,
then  $c_ix_0^i\in \cO_K[\tilde{c}x_0]$ by (F2). So let us assume that $vc_ix_0^i\in
\cC_{vL}(v\tilde{c}x_0)$. Denote by $\alpha$ the image of $v\tilde{c}x_0$ and by
$\beta$ the image of $vc_ix_0^i$ in $\cA_{vL}(v\tilde{c}x_0)$. Note that both of them
are positive, so
\[
-i\alpha\><\> \beta-i\alpha\>.
\]
By the density of $\cA_{vK}(qv\tilde{c}x_0)=\cA_{vK}(qvc_mx_0^m)$ in $\cA_{vL}
(vc_mx_0^m)=\cA_{vL}(v\tilde{c}x_0)$ there is $c_0\in K$ such
that the image $\gamma$ of $vc_0$ in $\cA_{vL}(v\tilde{c}x_0)$ satisfies
\[
-i\alpha\><\> i\gamma\><\> \beta-i\alpha\>,
\]
whence $0<i\gamma+i\alpha<\beta$.
This leads to $0<vc_0^i \tilde{c}^i x_0^i<vc_ix_0^i$. Setting $c=c_0\tilde{c}$, we
obtain that $0<vc^ix_0^i<vc_ix_0^i$, whence $c_ix_0^i\in (cx_0)^i\cO_K\subseteq 
\cO_K[cx_0]$ with $vcx_0>0$. We have proved that also in this case, (\ref{Acx}),
(\ref{(DL2b)A}) and (\ref{(DL2b)}) hold for $x=x_0\,$.

\parm
In case (DL2a), fact (F2) shows that (\ref{(DL2b)}) holds for $x=x_0$ because for every
$c_ix_0^i\in A_0$ there is some $\tilde{c}x_0\in A_0$ with $\cC_{vL}(v\tilde{c}x_0)
\subseteq \cC_{vL}^+(vc_ix_0^i)$. In case (DL2b) the latter is not true, but using
density we were able to show $c_ix_0^i\in\cO_K[cx_0]$ for some $c\in K$ with $vcx_0>0$
even when $vc_ix_0^i\in\cC_{vL}(v\tilde{c}x_0)$. In cases (DL2a) and (DL2b) we set
$x:=x_0\,$. The next case treats the instance where we do not have density at hand.

\mn
(DL2c): In this case, $\tilde H_\cE/H_\cE$ is discrete. Choose the element $c_mx_0^m$ as
in case (DL2b). Now we have that $\cA_{vL}(vc_mx_0^m)$ and $\cA_{vK}(qvc_mx_0^m)$ are
discrete. From (F1) we know that there is $\tilde{c}\in K$ such that $\tilde{c}x_0\in
A_0$ and $\cC_{vL}(v\tilde{c}x_0)=\cC_{vL}(vc_mx_0^m)$. The image $\alpha$ of
$v\tilde{c}x_0$ in $\cA_{vL}
(v\tilde{c}x_0)$ may not be its smallest positive element, which creates the problem that 
not all elements $c_ix_0^i\in A_0$ with $vc_ix_0^i\in \cC_{vL}(v\tilde{c}x_0)$ may lie in
$\cO_L[\tilde{c}x_0]$. So take any $c_jx_0^j\in A_0$ ($j\in\{1,\ldots,q-1\}$) with
$vc_jx_0^j\in\cC_{vL}(v\tilde{c}x_0)$ whose image $\gamma$ in $\cA_{vL}(v\tilde{c}x_0)$ 
is its smallest positive element. Since $j\in\{1,\ldots,q-1\}$, also
$1,x,\ldots,x^{q-1}$ form a valuation basis of $(L|K,v)$. Hence we may set $x:=x_0^j$ and
from now on work with $A_x$ in place of $A_0\,$.

Now we have that $vc_jx$ is the smallest positive element in $\cA_{vL}(vc_j x)=\cA_{vL}
(v\tilde{c}x_0)$ and $qvc_jx$ is the smallest positive element in $\cA_{vK}(qvc_j x)$. 
Further, the only elements strictly between $0$ and $q\gamma$ are $\gamma, 2\gamma,
\ldots,(q-1)\gamma$.

Take any element $c_ix^i\in A$. If $vc_ix^i\notin \cC_{vL}(vc_j x)$, then $c_ix^i\in 
(c_j x)^i\cO_K\subseteq\cO_K[c_j x]$ by (F2). So let us assume that $vc_ix^i\in\cC_{vL}
(vc_j x)$, and denote the image of $vc_ix^i$ in $\cA_{vL}
(vc_j x)$ by $\beta$. Write $c_ix^i=dc_j^ix^i$ with $d=c_ic_j^{-i}$, and denote 
by $\delta$ the image of $d$ in $\cA_{vL}(vc_j x)$, so that $0\leq\beta=\delta
+i\gamma$. Suppose that $\delta<0$; then $\delta+i\gamma=k\gamma$ for some $k\in
\{0,\ldots,i-1\}$, but as $\delta\in \cA_{vK}(qvc_j x)$, this is impossible. Hence
$\delta\geq 0$. If $\delta>0$, then $vd>0$, whence $c_ix^i\in (c_jx)^i\cO_K\subseteq
\cO_K[c_jx]$. 

Now assume that $\delta=0$. Then $vd\in \cC^+_{vL}(vc_j x)$.
If $d\in\cO_K\,$, then we are done again. So assume that $vd<0$ and write $c_ix^i=
dc_j^ix^i=d^{1-i}(dc_jx)^i$. Then $d^{1-i}\in\cO_K\,$, hence for $c:=dc_j\,$,
$c_ix^i\in (cx)^i\cO_K\subseteq \cO_K[cx]$. As $vd\in \cC^+_{vL}(vc_jx)$ and $vc_jx>0$, 
we have $vcx=vd+vc_jx>0$.

We have proved that in this case, (\ref{Acx}), (\ref{(DL2b)A}) and (\ref{(DL2b)}) hold for
$x=x_0^j\,$.

\pars
\begin{remark}
\rm Assume that $vK$ is $i$-divisible for all $i\in\{2,\ldots,q-1\}$, with $q$ not
necessarily prime. Take $c_ix_0^i\in A_0\,$. Then there is $c\in K$ such that $vc_i=ivc$.
We obtain that $vc^ix_0^i=vc_ix_0^i>0$, hence also $vcx_0>0$. Consequently, $c_ix_0^i\in
(cx_0)^i\cO_K^\times\subseteq\cO_K[cx_0]$. It follows that (\ref{Acx}), (\ref{(DL2b)})
and (\ref{Acx}) hold for $x=x_0\,$.

This case appears when $q=\chara Kv>0$ and $(K,v)$ is equal to its own absolute
ramification field, since then $vK$ is divisible by all primes other than $q$.
\end{remark}

\pars
Assume that $\cE$ is of type (DL2c) and, using the notation of that case, that $\cC^+_{vL}
(v\tilde{c}x_0)=\{0\}$ or equivalently, $H_\cE=\{0\}$. Then in the case of $\delta=0$
we have $vd=0$, whence $vc_ix^i=vc_j^ix^i$ and $c_ix^i\in (c_jx)^i\cO_K^\times\subseteq
\cO_K[c_jx]$. This shows that $\cO_L=\cO_K[c_jx]$. The assumption $H_\cE=\{0\}$ holds
in case (DL2c) if and only if $[L:K]=(vL:vK)$ equals the initial index of the extension
$(L|K,v)$, which is the number of nonnegative values of $vL$ that are smaller than any
positive element in $vK$. Therefore, our result is a proof of Knaf's conjecture about
essentially finite generation of $\cO_L$ over $\cO_K$
for the case of extensions of prime degree. The formulation and the (considerably more
involved) full proof of Knaf's conjecture
is given in \cite{D}. See also \cite{CuNo,No} for proofs of important special cases.

\pars
We summarize what we have shown in case (DL2):
\begin{theorem}                                \label{thmgen}
Take an extension $\cE=(L|K,v)$ of prime degree $q=\rme (L|K,v)$, with
$x_0\in L$ such that $vx_0\notin vK$.
\sn
1) If $\cE$ is of type (DL2a) or (DL2b), then (\ref{(DL2b)}) -
(\ref{Acx}) hold for $x=x_0\,$.
\sn
2) If $\cE$ is of type (DL2c), then (\ref{(DL2b)}) - (\ref{Acx}) hold for
$x=x_0^j$ with suitable $j\in\{1,\ldots,q-1\}$. If in addition $H_\cE=\{0\}$,
then $\cO_L=\cO_K[cx]$ for suitable $c\in K$.

\sn
The assumption of part 1) holds in particular when every archimedean component of $vK$
is dense, and this in turn holds for every deeply ramified field $(K,v)$.

\pars
In all cases, $1,x,\ldots,x^{q-1}$ form a valuation basis of $(L|K,v)$, and for all 
$c,c'\in K$,
\[
\cO_K[c x] \>\subseteq\> \cO_K[c' x] \>\Leftrightarrow vc\leq vc'\>.
\]
\end{theorem}
\begin{proof}
Only the implication ``$\Rightarrow$'' of the last assertion needs a proof.
Take $c,c'\in K$.
If $\cO_K[c x]
\subseteq\cO_K[c' x]$, then $c x\in \cO_K[c' x]$. Since the elements $1,c'x,\ldots,
(c'x)^{n-1}$ form a valuation basis of $(L|K,v)$, it follows that $c x=\frac{c}{c'}c' x$ 
with $\frac{c}{c'}\in\cO_K\,$, whence $vc\leq vc'\,$.
\end{proof}

\sn
%
%
\subsection{Valuation rings and ideals associated with the generation of $\cO_L|\cO_K$}\label{sectivr}
\mbox{ }\sn

The convex subgroup $H_\cE$ of $vL$ and the associated valuation ring and maximal ideal
turn out to be important invariants of the extension $\cE$. We take a closer look at them 
in this section.

The convex subgroups $H$ of $vL$ are in one-to-one correspondence with the coarsenings 
$v_H$ of $v$ on $L$ in such a way that $v_H L=vL/H$. The valuation ring of $v_H$ on $L$ 
is $\cO_{v_H}=\{a\in L \mid \exists\gamma\in H:\,va\geq\gamma\}$, and its maximal ideal 
is $\cM_{v_H}=\{a\in L \mid va>H\}$. We write $\cO_\cE$ for $\cO_{v_{H_\cE}}$, 
$\,\cM_\cE$ for $\cM_{v_{H_\cE}}$, and $v_\cE$ for~$v_{H_\cE}\,$. 

We note that $\cM_\cE$ is a nonprincipal $\cO_\cE$-ideal if $\cE$ is of type (DL2a) or
(DL2b), and a principal $\cO_\cE$-ideal if $\cE$ is of type (DL2c). Indeed, the value group
$v_\cE L$ is (up to equivalence) the quotient $vL/H_\cE$. In case (DL2a), this does not
have a smallest convex subgroup and thus no smallest positive element. In case (DL2b)
the quotient has a smallest convex subgroup. As it is dense, it does not have a
smallest positive element, and therefore the same holds for $v_\cE L$. Also in case (DL2c)
the quotient has a smallest convex subgroup. As now it is discrete, it has a smallest
positive element, and the same holds for $v_\cE L$.

\begin{proposition}                           \label{I}
Take an extension $\cE=(L|K,v)$ of prime degree $q=\rme (L|K,v)$, with $x$ determined by
Theorem~\ref{thmgen}. Then for every  $a\in L$ such that $va>H_\cE$ there is
$c\in K$ with $0<vcx\leq va$. Further, $\cM_\cE$ is equal to the $\cO_L$-ideal 
\begin{equation}                 \label{Ixdef}
I_x\>:=\> (cx\mid c\in K \mbox{ with } vcx>0)\>.
\end{equation}
\end{proposition}
\begin{proof}
Take $a\in L$ such that $va>H_\cE$ and write $a=\sum_{i=0}^{q-1} c_i x^i$ with $c_i\in K$.
Since $1,x,\ldots,x^{q-1}$ form a valuation basis, we have $vc_i x^i\geq va>0$ for 
$0\leq i\leq q-1$ with $va=\min_i vc_i x^i\,$.
In particular, $c_i x^i\in A_x$ for $1\leq i\leq q-1$. Hence it follows
from (\ref{Acx}) that for each such $i$ there is $d_i\in K$ with $0<vd_ix\leq v(d_ix)^i
\leq vc_ix^i$. It remains to consider the case of $i=0$. We have $vc_0\geq va>H_\cE\,$.

If $H_\cE\subsetneq\cC_{vL}^+(vc_0)$, then there is some $c_\ell x^\ell\in A_x$ with
$vc_\ell x^\ell\in \cC_{vL}^+(vc_0)$.
By (\ref{Acx}) there is $d_0\in K$ with $0<vd_0x\leq v(d_0x)^\ell\leq vc_\ell x^\ell
<vc_0$.

Now assume that $H_\cE=\cC_{vL}^+(vc_0)$, so $\cE$ is of type
(DL2b) or (DL2c). We will use the notation as in the computations for these two cases
in Section~\ref{sectgen}. With the element $c_mx_0^m$
appearing in these cases, we have $\cC_{vL}(vc_0)=\cC_{vL}
(vc_mx_0^m)$ and $\cA_{vL}(vc_0)=\cA_{vL}(vc_mx_0^m)$. 

In case (DL2b), $\cA_{vK}(qvc_mx_0^m)$
is dense in $\cA_{vL}(vc_mx_0^m)=\cA_{vL}(vc_0)$. Denote by $\alpha$ the image of 
$vc_mx_0^m$ and by $\gamma$ the image of $vc_0$ in $\cA_{vL}(vc_0)$. By density, there 
is $b\in K$ such that $\alpha-\gamma<\beta<\alpha$ with $\beta$ the image of 
$vb$ in $\cA_{vL}(vc_0)$. This leads to $vc_mx_0^m c_0^{-1}<vb<vc_mx_0^m\,$, that is,
$vc_0>vc_mx_0^m-vb>0$. Since $x=x_0$ in the present case, we obtain that
$b^{-1}c_mx_0^m=b^{-1}c_mx^m\in A_x\,$. Hence by (\ref{Acx}) there is
$d_0\in K$ with $0<vd_0x\leq v(d_0x)^m\leq vb^{-1}c_m x^m<vc_0$.

In case (DL2c), $\cA_{vK}(qvc_mx_0^m)$ is not dense in $\cA_{vL}(vc_mx_0^m)=
\cA_{vL}(vc_0)$. With $c_jx_0^j\in A_0$ chosen as in case (DL2c), the image $\gamma$ of
$vc_jx_0^j$ in $\cA_{vL}(vc_mx_0^m)$ is its smallest positive element. Then the image
$q\gamma$ of $qvc_jx_0^j$ is the smallest positive element of $\cA_{vK}(qvc_mx_0^m)
=\cA_{vK}(vc_0)$. Since $x=x_0^j$ in the present case, we obtain that $vc_0
\geq qvc_jx_0^j>vc_jx_0^j=vc_jx$, and we set $d_0=c_j\,$.

We have now proved that in all cases there is $d_0\in K$ such that $0<vd_0x<vc_0\,$.
We choose some $i_0\in \{0,\ldots,q-1\}$ such that $vd_{i_0}=\min \{vd_i\mid
0\leq i\leq q-1\}$ and set $c:=d_{i_0}$. Then
\begin{equation}                           \label{vcxleq}
vcx\>\leq\> vd_i x\>\leq\> vc_i x^i\quad\mbox{ for }\>0\leq i\leq
q-1\>.
\end{equation}
Hence $0<vcx\leq va$ as required.

\parm
Now we prove the second assertion.
All elements $cx$ as in (\ref{Ixdef}) lie in $A_x$ and therefore have value $>H_\cE$. It
follows that all elements in $I_x$ have value $>H_\cE$ and thus lie in $\cM_\cE\,$. This 
proves the inclusion $I_x\subseteq\cM_\cE\,$. 

For the converse, take $a\in\cM_\cE$, so $va>H_\cE\,$. By the first assertion of our
proposition, there is $c\in K$ with $0<vcx\leq va$. This implies $a\in cx\cO_L\subseteq 
I_x\,$.
\end{proof}

We give an application of this proposition.
\begin{corollary}                           
Take an extension $\cE=(L|K,v)$ of prime degree $q=\rme(L|K,v)$, with $x\in L$
determined by Theorem~\ref{thmgen}. If $\cE$ is of type (DL2a) or (DL2b), then
for every $a\in\cM_\cE$ there is $c\in K$ such that $a\,\cO_L\subseteq \cO_K[cx]$.
\end{corollary}
\begin{proof}
Since $a\in\cM_\cE\,$, we have $va>H_\cE\,$. In case (DL2a), $H_\cE\subsetneq
\cC_{vL}^+(va)$. Then there is an element $c_\ell x^\ell \in A_x\cap\cC_{vL}^+(va)$.
By Proposition~\ref{I}, there is $c\in K$ such that $0<vcx\leq vc_\ell x^\ell<va$. It 
follows that $vcx\in\cC_{vL}^+(va)$, hence $qvcx\in \cC_{vL}^+(va)$ and therefore, 
$qvcx\leq va$.

In case (DL2b), $H_\cE=\cC_{vL}^+(va)$ and $\cA_{vL}(va)$ is dense.
Hence there is $b\in L$ such that $H_\cE=\cC_{vL}^+(va)<vb$ (so $b\in
\cM_\cE$) and $qvb\leq va$. By Proposition~\ref{I} there is $c\in K$ such that
$0<vcx\leq vb$, whence again, $qvcx\leq va$.

Take any $a'\in a\,\cO_L$, so $va'\geq va$. Write $a'=\sum_{i=0}^{q-1} c_i x^i$. Then 
$vc_i x^i \geq va'\geq va\geq qvcx\geq vc^ix^i$ for $0<i<q-1$, hence $c_i x^i\in
c^ix^i\cO_K\subseteq \cO_K[cx]$. Since also $c_0\in\cO_K\subseteq \cO_K[cx]$, we obtain
that $a'\in\cO_K[cx]$, which shows that $a\,\cO_L\subseteq \cO_K[cx]$.
\end{proof}
\noindent
Note that the assertions of this corollary are trivially satisfied if $q=\rmf(L|K,v)$.
Moreover, the last assertion also holds if $\cE$ is of type (DL2c) with $H_\cE=\{0\}$.

\pars
The ideal $\cM_\cE$ will be useful in the computation of the K\"ahler differentials in
Theorems~\ref{OASe} and~\ref{OKume}. In preparation, we need a small technical lemma.

\begin{lemma}                     \label{aM^n}
Take a valuation ring $\cO$ with maximal ideal $\cM$. Whenever $2\leq n\in\N$ and
$a\in\cO_L$, then
\sn
1) \ $a\cM=\cM$ if and only if $a\notin\cM$,
\sn
2) \ $\cM^n=\cM$ if and only if $\cM$ is a nonprincipal $\cO$-ideal,
\sn
3) \ $(a\cM)^n= a\cM$ if and only if $a\notin\cM$ and $\cM$ is a nonprincipal $\cO$-ideal.
\end{lemma}
\begin{proof}
Denote by $w$ the valuation associated with $\cO$.
\sn
1): We have $a\notin a\cM$, hence if $a\in\cM$, then $a\cM\ne\cM$. If $a\notin\cM$, then
$a$ is a unit in $\cO$, so $a\cM=\cM$.
\sn
2): The value group of $w$ is not discrete, and hence dense, if and only if $\cM$ is a
nonprincipal $\cO$-ideal. If it is discrete and $\gamma$ is its smallest positive element,
then $\cM=\{b\in K\mid wb\geq\gamma\}$ and $\cM^n=\{c\in K\mid wc\geq n\gamma\}\subsetneq
\cM$ since $n\gamma>\gamma$. If it is dense, then for every $b\in\cM$ there is $c\in K$
such that $0<nwc<wb$, whence $b\in \cM^n$; therefore, $\cM^n\subseteq\cM\subseteq\cM^n$ and
consequently, $\cM^n=\cM$.
\sn
3): If $a\notin\cM$ and $\cM$ is a nonprincipal $\cO$-ideal, then by parts 1) and 2),
$(a\cM)^n= \cM^n=\cM=a\cM$.
If $a\in\cM$, then $wa>0$, whence $a\cM=\{c\in K\mid wc>wa\}$ and $(a\cM)^n\subseteq
a^n\cM=\{c\in K\mid wc>nwa\}\subsetneq a\cM$ since $nwa>wa$. If $\cM$ is a principal
$\cO$-ideal, say $\cM=b\cO$ with $b\in\cM$, then $a\cM=ab\cO=\{c\in K\mid wc\geq wab\}$
and $(a\cM)^n=(ab\cO)^n=\{c\in K\mid wc\geq nwab\}\subsetneq a\cM$ since $nwab>wab$ and
$ab\in a\cM$.
\end{proof}

\sn
%
%
\subsection{Differents of generators for Artin-Schreier and Kummer 
extensions}                                \label{sectdiff}
%
The proofs in Sections~\ref{sectseprfe} to~\ref{sectOKume} make use of the differents of 
the chosen generators for $\cO_L$ as an $\cO_K$-algebra. In this section we compute those
differents. 

If $b\in L$ and $h_b$ is its minimal polynomial over $K$, then $\delta(b):=h_b'(b)$ is
called the \bfind{different of $b$}. The $\cO_L$-ideal
\begin{equation}              \label{D0}
\cD_0(\cO_L|\cO_K)\>:=\>(h_b'(b)\mid b\in \cO_L\setminus\cO_K)\>.
\end{equation}
generated by the differents of all elements in $\cO_L\setminus\cO_K$ will be
called the \bfind{naive different ideal}.
%
\begin{proposition}                    \label{ndi}
Assume that $(L|K,v)$ is a nontrivial finite unibranched Galois extension and that
\[
\cO_L\>=\>\bigcup_{\alpha\in S} \cO_K[b_\alpha]
\]
for some (possibly finite) index set $S$ and elements $b_\alpha\in\cO_L\setminus
\cO_K\,$. Then $\cD_0(\cO_L|\cO_K)$ is equal to the $\cO_L$-ideal $(\delta(b_\alpha)
\mid\alpha\in S)$.
\end{proposition}
\begin{proof}
In the proof of \cite[Proposition~4.1]{CuKuRz} it is shown that $b\in\cO_K[b_\alpha]$
implies $v\delta(b)\geq v\delta(b_\alpha)$. Hence,
\[
(\delta(b) \mid b\in\cO_L\setminus\cO_K)\>=\>\bigcup_{\alpha\in S}
(\delta(b) \mid b\in \cO_K[b_\alpha]\setminus\cO_K)\>=\>\bigcup_{\alpha\in S}
(\delta (b_\alpha))\>=\> (\delta(b_\alpha)\mid\alpha\in S)\>.
\]
\end{proof}

\pars
In the case of Artin-Schreier and Kummer extensions $(L|K,v)$ with Galois group $G$ 
we have sufficient information about the minimal polynomials $f$ of the various 
generators $x$ we have worked with in the previous sections, and about
their conjugates, to work out the values $vf'(x)$ of their differents $f'(x)$. In order 
to do this, we can either compute $f'$, or we can use the formula
\begin{equation}                              \label{vf'}
f'(x)\>=\>\prod_{\sigma\in G\setminus\{\mbox{\tiny\rm id}\}} (x-\sigma x)\>.
\end{equation}

We keep the notations from the previous sections.

\mn
%
%
\subsubsection{Artin-Schreier extensions}
\mbox{ }\sn
Take an Artin-Schreier polynomial $f$ with $\vartheta$ as its root. Then its minimal
polynomial is $f(X)=X^p-X-\vartheta^p+\vartheta$ with $f'(X)=-1$, whence 
\begin{equation}                 \label{vf'AS}
f'(\vartheta)=-1\>.
\end{equation} 

\pars
For $c\in K^\times$, denote by $f_c$ the minimal polynomial of $c\vartheta$. Then
\begin{equation}                                 \label{vg'}
f_c'(c\vartheta)\>=\>\prod_{\sigma\in G\setminus\{\mbox{\tiny\rm id}\}} 
(c\vartheta-\sigma c\vartheta)\>=\>c^{p-1}f'(\vartheta)\>=\>-c^{p-1}\>. 
\end{equation}

\begin{lemma}                               \label{duASpa}
Take an Artin-Schreier extension $\cE=(L|K,v)$ of prime degree $p=\rmf (L|K,v)$. 
If the extension $Lv|Kv$ is purely inseparable, then $\cE$ admits an
Artin-Schreier generator $\vartheta$ of value $v\vartheta<0$ and $\tilde{c}\in K$
such that $v\tilde{c}\vartheta=0$, $Lv=Kv(\tilde{c}\vartheta v)$, $\cO_L= 
\cO_K[\tilde{c}\vartheta]$ and 
\begin{equation}                               \label{duASpadisp}
\cD_0(\cO_L|\cO_K)\>=\>f_{\tilde{c}}'(\tilde{c}\vartheta)\cO_L\>=\>\tilde{c}^{p-1}\cO_L
\>=\> I_\cE^{p-1} \>.
\end{equation}
\end{lemma}
\begin{proof}
The first assertions follow from part 1) of Proposition~\ref{ASKgen} and case (DL1). 
Applying Proposition~\ref{ndi} with $S=\{1\}$ and $b_1=\tilde{c}\vartheta$, we obtain
the first equality of (\ref{duASpadisp}).
Since $v\tilde{c}=-v\vartheta$, we have $\cD_0(\cO_L|\cO_K)=
f_{\tilde{c}}'(\tilde{c}\vartheta)\cO_L=\tilde{c}^{p-1}\cO_L=
(\vartheta^{-1})^{p-1}=I_\cE^{p-1}$ by (\ref{vg'}) and part 1) of
Proposition~\ref{PropASKgen}. This proves (\ref{duASpadisp}).
\end{proof}

\begin{lemma}                               \label{duASpb}
Take an Artin-Schreier extension $\cE=(L|K,v)$ of prime degree $p=\rme (L|K,v)$. 
Then $\cE$ admits an Artin-Schreier generator $\vartheta$ of value $v\vartheta<0$ 
such that $vL=vK+\Z v\vartheta$, (\ref{(DL2b)}) holds for $x=\vartheta^j$ 
with suitable $j\in \{1,\ldots,p-1\}$, and we have the equality of $\cO_L$-ideals
\begin{equation}                      \label{hjc'cth^j}
\cD_0(\cO_L|\cO_K)\>=\>
(\vartheta^{-1}I_{\vartheta^j})^{p-1}\>=\>(I_\cE\cM_\cE)^{p-1} \>.
\end{equation}
\end{lemma}
\begin{proof}
The existence of such $\vartheta$ and $j$ follows from part 1) of
Proposition~\ref{ASKgen} together with Theorem~\ref{thmgen}. Since
(\ref{(DL2b)}) holds for $x=\vartheta^j$, we can apply Proposition~\ref{ndi} with
$S=\{c\in K^\times\mid vc\vartheta^j>0\}$ and $b_c=c\vartheta^j$ to obtain:
\[
\cD_0(\cO_L|\cO_K)\>=\>(h_{j,c}'(c\vartheta^j)\mid c\in K^\times\mbox{ with }
vc\vartheta^j>0)\>,
\]
where $h_{j,c}$ denotes the minimal polynomial of $c\vartheta^j$. Now we compute:
\[
c\vartheta^j-\sigma c\vartheta^j\>=\>c(\vartheta^j -(\sigma\vartheta)^j)\>=\>
c(\vartheta^j -(\vartheta+k)^j)\>=\> -c\sum_{i=1}^j \binom{j}{i}\vartheta^{j-i} k^i
\]
for suitable $k\in\F_p^\times$. The summand of least value in the sum on 
the right hand side is the one for $i=1$. Using (\ref{vf'}), we obtain:
\begin{equation}                                \label{vh'}
vh_{j,c}'(c\vartheta^j)\>=\> (p-1)(vc\vartheta^{j-1})\>.
\end{equation}
Hence,
\begin{equation*}
\cD_0(\cO_L|\cO_K)\>=\>(c\vartheta^{j-1}\mid vc\vartheta^j>0)^{p-1}\>=\>
(\vartheta^{-1} I_{\vartheta^j})^{p-1}\>=\> (I_\cE\cM_\cE)^{p-1} \>,
\end{equation*}
where the last two equalities follow from
part 1) of Proposition~\ref{PropASKgen} and Proposition~\ref{I}.
This proves (\ref{hjc'cth^j}).
\end{proof}

\mn
%
%
\subsubsection{Kummer extensions}
\mbox{ }\sn
In what follows, $\zeta_p$ will denote a primitive $p$-th root of unity. If $L|K$ is 
a Kummer extension, then $\zeta_p\in K$.
Take a Kummer polynomial $f$ of degree $q$ with $\eta$ as its root. Then $f(X)=
X^q-\eta^q$ and $f'(X)=qX^{q-1}$, whence 
\begin{equation}                    \label{vf'eta}
f'(\eta)\>=\>q\eta^{q-1}\>.
\end{equation}

\begin{lemma}                                     \label{duKpf}
Take a Kummer extension $\cE=(L|K,v)$ of degree $p=\chara Kv$. Assume that
$\rmf (L|K,v)=p$. Then there exists a Kummer generator $\eta\in L$ such that 
one of the following cases holds:
\sn
i) $v\eta=0$, $Lv=Kv(\eta v)$ with $Lv|Kv$ inseparable, and $\cO_L=\cO_K[\eta]$, 
\sn
ii) $\eta$ is a 1-unit, $v\tilde{c}(\eta-1)=0$, $Lv=Kv(\tilde{c}(\eta-1)v)$ and 
$\cO_L=\cO_K[\tilde{c}(\eta-1)]$ for suitable $\tilde{c}\in K^\times$. 
\sn
In case i), for $f$ the minimal polynomial of $\eta$, 
\begin{equation}                         \label{f'etaO}
\cD_0(\cO_L|\cO_K)\>=\>f'(\eta)\cO_L\>=\>p\cO_L\>=\> I_\cE^{p-1}\>.
\end{equation}

In case ii), for $h_{\tilde{c}}$ the minimal polynomial of $\tilde{c}(\eta-1)$, 
\begin{equation}                         \label{h_tcO}
\cD_0(\cO_L|\cO_K)\>=\> h_{\tilde{c}}'(\tilde{c}(\eta-1))\cO_L\>=\> p\tilde{c}^{p-1}\cO_L
\>=\> I_\cE^{p-1}\>.
\end{equation}
%
\end{lemma}
\begin{proof}
The existence of such $\eta$ and $\tilde{c}$ follows from part 2)a) of 
Proposition~\ref{ASKgen}. The presentation of $\cO_L$ follows from case (DL1).
Applying Proposition~\ref{ndi} with $S=\{1\}$ and setting $b_1=\eta$ and $b_1=\tilde{c}
(\eta-1)$, respectively, we obtain the first equalities of (\ref{f'etaO}) and
(\ref{h_tcO}).

In case i), the second equality of 
(\ref{f'etaO}) follows from (\ref{vf'eta}) since $v\eta=0$. The third equality holds since
$vp=(p-1)v(\zeta_p-1)$ by (\ref{vz-1}), whence $p\cO_L=I_\cE^{p-1}$ by part 2)a) of 
Proposition~\ref{PropASKgen}.

\pars
For case ii) we compute with $\sigma$ a generator of $\Gal L|K$, using (\ref{vf'}):
\[
h_{\tilde{c}}'(\tilde{c}(\eta-1))\>=\> \prod_{i=1}^{p-1} \tilde{c}(\eta -\sigma^i\eta )
\>=\> (\tilde{c}\eta)^{p-1}\prod_{i=1}^{p-1}(1-\zeta_p^i)\>,
\]
whence by (\ref{prod1-z}),
\begin{equation}                                \label{vh'2}
vh_{\tilde{c}}'(\tilde{c}(\eta-1))\>=\> vp(\tilde{c}\eta)^{p-1}\>.
\end{equation}
This yields the second equality of (\ref{h_tcO}) since $v\eta=0$. The third holds as
$v\tilde{c}=-v(\eta-1)$ yields $vp\tilde{c}^{p-1}=(p-1)(v(\zeta_p-1)-v(\eta-1))$, whence
$p\tilde{c}^{p-1}\cO_L=I_\cE^{p-1}$ by part 2)b) of 
Proposition~\ref{PropASKgen}.
\end{proof}

\begin{lemma}                                     \label{duKpe}
Take a Kummer extension $\cE=(L|K,v)$ of prime degree $q=\rme (L|K,v)$. Then there are
two possible cases.
\sn
i) There is a Kummer generator $\eta\in L$ such that $vL=vK+\Z v\eta$,  
(\ref{(DL2b)}) holds for $x=\eta$, and we have the equality 
\begin{equation}                         \label{fc'etaI}
\cD_0(\cO_L|\cO_K)\>=\>qI_\eta^{q-1} \>=\>q\cM_\cE^{q-1}\>
\end{equation}
of $\cO_L$-ideals. If $q=\chara Kv$, then 
\begin{equation}                         \label{fc'etaIp}
\cD_0(\cO_L|\cO_K)\>=\>(I_\cE\cM_\cE)^{q-1}\>.
\end{equation}

If $q\ne\chara Kv$, then always this case i) holds, and the factor $q$ can be dropped
in (\ref{fc'etaI}) since $vq=0$.

\mn
ii) There is a Kummer generator $\eta\in L$ which is a $1$-unit such that for 
\begin{equation}                      \label{xi}
\xi\>:=\> \frac{\eta-1}{\zeta_q-1}\>,
\end{equation}
we have that $v\xi<0$, $vL=vK+\Z v\xi$, (\ref{(DL2b)}) holds for $x=\xi^j$ 
with suitable $j\in \{1,\ldots,q-1\}$, and we have the equality of $\cO_L$-ideals
\begin{equation}                         \label{hc'xiI}
\cD_0(\cO_L|\cO_K)\>=\>(\xi^{-1}I_{\xi^j})^{q-1} \>=\>(I_\cE\cM_\cE)^{q-1}\>.
\end{equation}
\end{lemma}
\begin{proof}
By part 2) of Proposition~\ref{trKext} and part 2)b) of Proposition~\ref{ASKgen}, the
extension admits a Kummer generator $\eta$ such that either $v\eta$ 
generates the value group extension, or $\eta$ is a $1$-unit and $v(\eta-1)$ generates 
the value group extension; moreover, the first case always holds if $q\ne\chara Kv$. 

Let us consider the first case.
Applying Theorem~\ref{thmgen} with $x_0=\eta$, we find that (\ref{(DL2b)}) holds for
$x=\eta^j$ with suitable $j\in \{1,\ldots,q-1\}$.  Since $\eta^j$ is
again a Kummer generator and also $v\eta^j$ generates the value group
extension as $j$ is prime to $q$, we may replace $\eta$ by $\eta^j$. As now
(\ref{(DL2b)}) holds for $x=\eta$, we can apply Proposition~\ref{ndi} with
$S=\{c\in K^\times\mid vc\eta>0\}$ and $b_c=c\eta$ to obtain:
\[
\cD_0(\cO_L|\cO_K)\>=\>(f_c'(c\eta)\mid c\in K^\times\mbox{ with } vc\eta>0)\>,
\]
where $f_c$ denotes the minimal polynomial of $c\eta$.

As also $c\eta$ is a Kummer generator, we can apply equation (\ref{vf'eta}) to obtain
that $f_c'(c\eta)=q(c\eta)^{q-1}$. Hence,
\begin{equation*}
\cD_0(\cO_L|\cO_K)\>=\> q(c\eta\mid c\in K \mbox{ with } vc\eta>0)^{q-1}\>=\>
q I_{\eta}^{q-1}\>=\> q\cM_\cE^{q-1}
\end{equation*}
where the last equation follows from Proposition~\ref{I}. This proves (\ref{fc'etaI}).

If $q=\chara Kv$, then $q\cM_\cE^{q-1}=((\zeta_q-1)\cM_\cE)^{q-1}=
(I_\cE\cM_\cE)^{q-1}$ since $vq=(q-1)v(\zeta_q-1)$ by (\ref{vz-1}) and the last 
equality follows from part 2)a) of Proposition~\ref{PropASKgen}. This proves
(\ref{fc'etaIp}).

\parm
Now we consider the second case. Since $L|K$ is a Kummer extension, $K$ contains 
$\zeta_q\,$. By Lemma~\ref{v(eta-1)}, $v(\eta-1)\leq v(\zeta_q-1)\in vK$ because $v\eta=0$.
Since $v(\eta-1)\notin vK$, inequality must hold. Hence with $\xi$ defined by 
(\ref{xi}), we have $v\xi<0$. Further, applying Theorem~\ref{thmgen} with $x_0=\xi$,
we find that (\ref{(DL2b)}) holds for $x=\xi^j$ with suitable $j\in \{1,\ldots,q-1\}$.
We apply Proposition~\ref{ndi} with $S=\{c\in K^\times\mid vc\xi^j>0\}$ and $b_c=c\xi^j$
to obtain:
\[
\cD_0(\cO_L|\cO_K)\>=\>(h_{j,c}'(c\xi^j)\mid c\in K^\times\mbox{ with } vc\xi^j>0)\>,
\]
where $h_{j,c}$ denotes the minimal polynomial of $c\xi^j$.

\pars
We note that $v(1-\zeta_q)=v(1-\zeta)$ for each primitive $q$-th root of unity $\zeta$.
We set $a:=\eta-1$. Then for every $\sigma\in G$, $v(a-\sigma a)=v(\eta-\sigma\eta)=
v(1-\zeta_q)>va$, hence
\[
a^j-\sigma a^j\>=\> a^j-(\sigma a)^j\>=\> a^j-(a+\sigma a -a)^j\>=\> 
-\sum_{i=0}^{j-1} \binom{j}{i} a^i (\sigma a -a)^{j-i}\>.
\]
Since $va<v(\sigma a -a)$, the summand of least value in the sum on the right 
hand side is the one for $i=j-1$. Consequently,
\begin{eqnarray*}
v(\xi^j-\sigma \xi^j) &=& v(a^j-\sigma a^j) -jv(1-\zeta_q)\>=\> (j-1)va+v(a-\sigma a) 
-jv(1-\zeta_q)\\
&=& (j-1)va+v(1-\zeta_q)-jv(1-\zeta_q)\>=\> (j-1)(va-v(1-\zeta_q))\\
&=& v\xi^{j-1}\>.
\end{eqnarray*}
Hence, equation (\ref{vf'}) shows that
\begin{equation}                                \label{vhc'K}
vh_{j,c}'(c\xi^j)\>=\> (q-1)vc\xi^{j-1}\>.
\end{equation}
Hence,
\begin{equation*}
\cD_0(\cO_L|\cO_K)\>=\> (c\xi^{j-1}\mid vc\xi^j>0)^{q-1}\>=\>
(\xi^{-1} I_{\xi^j})^{q-1}\>=\> (I_\cE\cM_\cE)^{q-1} \>
\end{equation*}
where the last two equalities follow from
by part 2)b) of Proposition~\ref{PropASKgen} and Proposition~\ref{I}.
This proves equation (\ref{hc'xiI}). 
%
\end{proof}

\bn
%
%
\section{K\"ahler differentials for Galois extensions of prime 
degree}                                          \label{sectcalc}
%
%
%
%
\subsection{Motivation} 
\mbox{ }\sn
We prove a proposition that will be a main tool for our handling of K\"ahler
differentials in the subsequent sections. It will provide a motivation for the 
calculation of the K\"ahler differentials for Artin-Schreier extensions and Kummer
extensions of prime degree which will be dealt with in this section.
\pars
Given a Galois extension $(L|K,v)$, we denote by $(L|K,v)\ina$ its inertia field
(cf.\ \cite[Section 19]{En}). 
%
\begin{proposition}                        \label{towerprop}
Let $(L|K,v)$ be
 a finite Galois extension. Then the following assertions hold.
\sn
1) There exists a tower of field extensions
\begin{equation}             \label{tow1}
K\>\subset\>K\ina\>=\>K_0\>\subset\> K_1\>\subset\,\cdots\,\subset\> K_\ell\>=\>L 
\end{equation}
where $K\ina=(L|K,v)\ina$ and each extension $K_{i+1}|K_i$ is a Galois extension of prime
degree. Note that if $K$ is henselian, then the extension $K\ina v|Kv$ is separable of
degree equal to $[K\ina:K]$.

\sn
2) Further, $(L|K,v)$ can be embedded in a finite 
Galois extension $(M|K,v)$ having the following properties:
\begin{equation}                         \label{tower}
\left\{\begin{array}{l}
\mbox{there exists a tower of field extensions}\\ 
\qquad\qquad\qquad\quad K\subset M_0\subset M_1\subset\cdots\subset M_m=M,\\ 
\mbox{where $M_0=(M|K,v)\ina$}\\ 
\mbox{and each extension $M_{i+1}|M_i$ is a Kummer extension of prime degree,}\\ 
\mbox{or an Artin-Schreier extension if the extension is of degree $p=\chara K$.}
\end{array}\right.
\end{equation}
\end{proposition}
\begin{proof} 
1): Set $K_0:=K\ina:=(L|K,v)\ina$. Since the extension $L|K\ina$ is solvable (cf.\ 
Theorems 24 
and 25 on pages 77 and 78 of \cite{ZS2}), there exists a tower (\ref{tow1}) of Galois
extensions such that each extension $K_{i+1}|K_i$ is Galois of prime degree. The 
assertions about the extension $K\ina v|Kv$ are part of the general properties of inertia
fields.
\sn
2): This proof is essentially the same argument as in the Galois characterization of sovability by radicals. If an extension $K_{i+1}|K_i$ in the tower (\ref{tow1}) is of degree $p=\chara K$, 
then it is an Artin-Schreier extension. If it is of prime degree $q\ne\chara K$, it is 
a Kummer extension if $K_i$ contains a primitive $q$-th root of unity. We will now 
explain how to enlarge the extension $(L|K,v)$ so that this will be the case for each
extension of prime degree $q\ne\chara K$ in a resulting new tower.

\pars
Assume that $(K,v)$ is of characteristic $0$ with $\chara Kv=p>0$ and that some 
extension $K_{i+1}|K_i$ is Galois of degree $p$, but $K$ does not contain a 
primitive $p$-th root of unity. In this case we will have to replace tower 
(\ref{tow1}) by a larger one. Let $\zeta_p$ denote a primitive $p$-th root of unity.
Then $K(\zeta_p)|K$ is a Galois extension, and so is $L(\zeta_p)|K$ since $L|K$ is
assumed to be Galois.
%

Set $K'_0:=(L(\zeta_p)|K,v)\ina$; then $K_0=K\ina\subset K'_0\,$. As before, $K'_0|K$ 
is Galois, hence so are $K'_0(\zeta_p)|K$ and $K'_0(\zeta_p)|K'_0\,$. By part 1) of our
proposition, there exists a tower of Galois extensions $K'_0\subset K'_1\subset\cdots
\subset K'_{r'}=K'_0(\zeta_p)$ such that each extension $K'_{i+1}|K'_i$ is Galois 
of prime degree. Since $[K'_0(\zeta_p):K'_0]<p$, none of the Galois extensions
$K'_{i+1}|K'_i$ is of degree $p$. 
\pars
We replace the tower (\ref{tow1}) by the tower
\begin{equation}             \label{tow2}
K'_0\,\subset\, K'_1\,\subset\cdots\subset\, K'_{r'}\>=\>K'_0(\zeta_p)\,\subset\,
K_1(\zeta_p)\,\subset\cdots\subset\, K_{\ell}(\zeta_p)\>=\>L(\zeta_p)\>.
\end{equation}
Now we have that if in mixed characteristic any extension in the tower (\ref{tow1}) is
Galois of degree $p=\chara Kv$, then it is a Kummer extension.

\pars 
We now return to the general case, with no restriction on the characteristic of $K$, first 
making the above change if necessary.

\pars
In order to also make sure that  all Galois extensions of prime degree $q\ne p$ in the
tower are Kummer extensions, we take $Q$ to be the set consisting of all such primes 
$q$. For every $q\in Q$, we choose a primitive $q$-th root of unity $\zeta_q$
and set $M:=L(\zeta_q\mid q\in Q)$. Every extension $K(\zeta_q)|K$ is
Galois, so $M|K$ is also a Galois extension.

Let us show that for every $q\in Q$, $\zeta_q$ lies in the inertia field of $(M|K,v)$. This is a standard fact, but we give a proof for completeness. 
The reduction of $X^q-1$ modulo $v$ is $X^q-1v$ with $1v$ being the $1$ in $Kv$. 
Since $q\ne\chara Kv$, the polynomial $X^q-1v$ has $q$ distinct roots. The minimal 
polynomial $f$ of $\zeta_q$ over $K$ divides $X^q-1$, so its reduction $fv$ divides
$X^q-1v$ and has therefore only simple roots. It follows that if $\sigma\in\Gal M|K$
with $\sigma\zeta_q\ne\zeta_q\,$, then $(\sigma\zeta_q)v\ne\zeta_q v$, whence 
$v(\sigma\zeta_q-\zeta_q)=0$. Hence every automorphism in the inertia group
$\{\sigma\in\Gal M|K\mid \forall x\in \cO_M: v(\sigma x-x)=0\}$ must fix $\zeta_q\,$,
which proves our claim. It follows that $M_0:=K_0(\zeta_q\mid 
q\in Q)$ is the inertia field of $(M|K,v)$. Finally, we set $M_i:=
K_i(\zeta_q\mid q\in Q)$. By our construction, now also all extensions of prime degree
$q\ne p$ in the tower are Kummer extensions. So we have obtained a tower as described 
in (\ref{tower}).
\end{proof}

\sn
%
%
\subsection{Some calculations of K\"ahler differentials} 
\mbox{ }\sn
Let $L|K$ be an algebraic field extension. Let $A\subseteq K$ be a normal 
domain whose quotient field is $K$. Assume that $z\in L$ is integral over $A$ and 
let $f(X)$ be the minimal polynomial of $z$ over $K$. Then $f(X)\in A[X]$ 
(see \cite[Theorem 4, page 260]{ZSI}). Since $f(X)$ is monic, $(f(X)K[X])\cap A[X]=
f(X)A[X]$, so $A[z]\cong A[X]/(f(X))$. Thus,
\begin{equation}\label{Kahlercomp}
\Omega_{A[z]|A}\>\cong\> [A[X]/(f(X),f'(X))]dX \>\cong\> [A[z]/(f'(z))]dX
\end{equation}
by \cite[Example 26.J, page 189]{Mat} and \cite[Theorem 58, page 187]{Mat}. There is a
universal derivation 
$d_{A[z]|A}:A[z]\rightarrow \Omega_{A[z]|A}$ defined by 
\begin{equation}\label{Simpd}
g(z)\mapsto 
[g'(z)]dX\mbox{ for }g(X)\in A[X],
\end{equation}
where $[g'(z)]$ is the class of $g'(z)$ in $A[z]/(f'(z))$.

We will also require the following theorem to calculate K\"ahler differentials. 
\sn 
\begin{proposition}{(\cite[Theorem 1.1]{CuKuRz})}                    \label{LimProp} 
Take an algebraic field extension $L|K$ of degree $n$, a normal domain $A$ with 
quotient field $K$ and a domain $B$ with
quotient field $L$ such that $B|A$ is an integral extension. Assume that there 
exist generators $b_{\alpha}\in B$ of $L|K$, which are indexed by a totally ordered set 
$S$, such that $A[b_{\alpha}]\subset A[b_{\beta}]$ if $\alpha\le\beta$ and
\begin{equation}                              \label{union}
\bigcup_{\alpha\in S}A[b_{\alpha}]\>=\>B.
\end{equation}
Further assume that there exist $a_{\alpha},a_\beta\in A$ such that $a_{\beta}\mid
a_{\alpha}$ if $\alpha\le \beta$ and for $\alpha\le\beta$, there exist $c_{\alpha,\beta}
\in A$ and expressions 
\begin{equation}                              \label{b_ab_b}
b_{\alpha}\>=\>\frac{a_{\alpha}}{a_{\beta}}b_{\beta}+c_{\alpha,\beta}\>.
\end{equation}
Let $h_{\alpha}$ be the minimal polynomial of $b_{\alpha}$ over $K$. Take $U$ 
and $V$ to be the $B$-ideals 
\begin{equation}                      \label{UV}
U\>=\>(a_{\alpha}\mid \alpha\in S)\quad\mbox{ and }\quad
V\>=\>(h_{\alpha}'(b_{\alpha})\mid \alpha\in S)\>.
\end{equation}

Then we have a $B$-module isomorphism
\begin{equation}                      \label{U/UV}
\Omega_{B|A}\>\cong\> U/UV\>.
\end{equation}

For the case where $(L|K,v)$ is a valued field extension and $A=\cO_K$ and $B=\cO_L\,$, 
for arbitrary $\gamma\in S$ the isomorphism (\ref{U/UV}) yields an $\cO_L$-module
isomorphism
\begin{equation}                 \label{altpres}
\Omega_{\cO_L|\cO_K}\>\cong\> U/b_\gamma^\dagger U^n \quad \mbox{ with }\quad
b_\gamma^\dagger\,:=\, \frac{h'_\gamma(b_{\gamma})}{a_\gamma^{n-1}}\>.
\end{equation}
\end{proposition}

Identifying the $B$-module $\Omega_{B/A}$ with $U/UV$ in the above theorem, the universal
derivation $d_{B|A}:B\rightarrow \Omega_{B|A}$ is defined by
\begin{equation}                                        \label{Compd}
d_{B|A}(z)\>=\>[a_{\alpha}g'_{\alpha}(b_{\alpha})]\in U/UV\quad\mbox{ for }\> z\>=\>
g_{\alpha}(b_{\alpha})\in A[b_{\alpha}]
\end{equation}
where  $[a_{\alpha}g'_{\alpha}(b_{\alpha})]$ is the class of $a_{\alpha}g'_{\alpha}
(b_{\alpha})$ in $U/UV$.

\parm
By definition, $V$ is the $B$-ideal generated by the differents of all $b_\alpha\,$. For
the case where $(L|K,v)$ is a valued field extension, we obtain from Proposition~\ref{ndi}:
\begin{lemma}
Under the assumptions of Proposition~\ref{LimProp}, the $\cO_L$-ideal $V$ defined in
(\ref{UV}) is equal to the $\cO_L$-ideal $\cD_0(\cO_L|\cO_K)$.
\end{lemma}
\noindent
This will be applied in the proofs of Theorems~\ref{OASe} and~\ref{OKume}.

\sn
%
%
\subsection{Finite extensions $(L|K,v)$ of degree  $[L:K]=\rmf(L|K)$
with separable residue field extension}                \label{sectseprfe}
\mbox{ }\sn
\begin{theorem}                 \label{seprfe}
Take a finite extension $(L|K,v)$  with $Lv|Kv$ separable of degree $[Lv:Kv]=
[L:K]$. Then $\cO_L=\cO_K[x]$ for some $x\in L$ with $vx=0$ and $Lv=Kv(xv)$, and we have
\[ 
\Omega_{\cO_L|\cO_K}=0\>.
\]
\end{theorem}
\begin{proof}
By (\ref{(DL2a)}), $\cO_{L}=\cO_K[x]$ where $x$ is a lift of a generator $\chi$ of $Lv$
over $Kv$. Let $f(X) \in K[X]$ be the minimal polynomial of $x$ over $K$. Since $vx=0$
and the extension is unibranched, also the conjugates of $x$ have value $0$ and thus, 
$f$ has coefficients in $\cO_K\,$. As $\deg f=
[L:K]=[Lv:Kv]$, the reduction $\bar f$ of $f$ in $Kv[X]$ is the minimal polynomial 
of $\chi$ over $Kv$. We have that $f'(x)v=\bar f'(\chi)$ which is nonzero since $\chi$ 
is separable over $Kv$. Thus $f'(x)$ is a unit in $\cO_{L}\,$. By (\ref{Kahlercomp}),
$\Omega_{{\cO_{L}} | \cO_K}\cong \cO_{L}/(f'(x))=0$.
\end{proof}

We note that this theorem always applies when $(L|K,v)$ is a Kummer extension of prime 
degree $q=\rmf(L|K)\ne \chara Kv$ since then $Lv|Kv$ is separable.

\sn
%
%
\subsection{Artin-Schreier extensions $(L|K,v)$ of degree $p$ with $\rmf(L|K)=p$ and
inseparable residue field extension}                \label{sectASextf=p}
\mbox{ }\sn
\begin{theorem}                      \label{OASf}
Take an Artin-Schreier extension $\cE=(L|K,v)$ of degree $p=\rmf(L|K)=\chara K$ with
$Lv|Kv$ inseparable. Then there exists an Artin-Schreier generator $\vartheta$ as in
Lemma~\ref{duASpa}, and we have
\begin{equation}                     \label{OASfdisp}
\Omega_{\cO_L|\cO_K}\>\cong\> \cO_L/(\tilde c^{p-1})\>\cong\>I_\cE/I_\cE^p
\end{equation}
as $\cO_L$-modules. Consequently, $\Omega_{\cO_L|\cO_K}\ne 0$.
\end{theorem}
\begin{proof}
The first isomorphism in
(\ref{OASfdisp}) follows from (\ref{Kahlercomp}) together with Lemma~\ref{duASpa}.
Since $v\tilde{c}=-v\vartheta$, we have $v\tilde{c}\ne 0$, whence $(\tilde c^{p-1})\ne
\cO_L\,$, as well as $\cO_L/(\tilde c^{p-1})\cong (\tilde c)/(\tilde c)^p=
I_\cE/I_\cE^p\,$ by part 1) of Proposition~\ref{PropASKgen}. This proves the second 
isomorphism in (\ref{OASfdisp}).
\end{proof}

With the notation of the statement and proof of Theorem \ref{OASf}, we have that for $z\in
\mathcal O_L$, $z=g(\tilde c \vartheta)$ for some $g(X)\in \mathcal O_K[X]$, and the 
universal derivation $d_{\mathcal O_L|\mathcal O_K}$ is defined by 
\[
d_{\mathcal O_L|\mathcal O_K}(z)\>=\>[g'(\tilde c\vartheta)]\in \mathcal O_L/(\tilde c^{p-1})
\]
by equation (\ref{Simpd}).

\sn
%
%
\subsection{Artin-Schreier extensions $(L|K,v)$ of degree $p$ with
$\rme(L|K,v)=p$}                      
\mbox{ }\sn
\begin{theorem}                       \label{OASe}
Take an Artin-Schreier extension $\cE=(L|K,v)$ of degree $p=\rme(L|K)$. Then there 
exists an Artin-Schreier generator $\vartheta$ as in Lemma~\ref{duASpb}, and we have
\begin{equation}                \label{Ome=p}
\Omega_{\cO_L|\cO_K}\>\cong\> \vartheta^{-1}\cM_\cE/(\vartheta^{-1}\cM_\cE)^p
\>=\> I_\cE\cM_\cE/(I_\cE\cM_\cE)^p
\end{equation}
as $\cO_L$-modules; in particular, $\Omega_{\cO_L|\cO_K}\ne 0$.
\end{theorem}
\begin{proof}
We will apply Proposition~\ref{LimProp} with $A=\cO_K$ and $B=\cO_L\,$. We set $S=
\{\alpha\in vK\mid \alpha+v\vartheta^j>0\}$, endowed with the reverse ordering of $vK$.
For each $\alpha\in S$ we choose $c_\alpha\in K$ such that $vc_\alpha=\alpha$. We
set $b_{\alpha}=c_\alpha\vartheta^j$, $a_{\alpha}=c_\alpha$, and $c_{\alpha,\beta}=0$.
Then $a_\beta | a_{\alpha}$ and $A[b_\alpha]\subseteq A[b_\beta]$ if $\alpha\leq\beta$,
and we have that $c_1\vartheta^j=\frac{c_1}{c_2}c_2\vartheta^j$.
We denote by $h_{\alpha}$ the minimal polynomial of $b_{\alpha}=c_\alpha\vartheta^j$
over $K$. Thus in the notation of Lemma~\ref{duASpb}, $h_{\alpha}=h_{j,\alpha}$ so that
$h_{\alpha}'(b_\alpha)=h_{j,\alpha}'(c_\alpha\vartheta^j)$ and $V=\cD_0(\cO_L|\cO_K)
=(\vartheta^{-1}I_{\vartheta^j})^{p-1}$ by equation
(\ref{hjc'cth^j}) of Lemma~\ref{duASpb}. Further, $U=(a_\alpha\mid\alpha\in S)=
(c_\alpha\mid\alpha\in S)=\vartheta^{-j}I_{\vartheta^j}\,$. Hence by 
Proposition~\ref{LimProp}, 
\[
\Omega_{\cO_L|\cO_K}\cong U/UV\>\cong\> \vartheta^{-j}I_{\vartheta^j}/\vartheta^{-j}
I_{\vartheta^j}\,(\vartheta^{-1}I_{\vartheta^j})^{p-1}\>\cong\> \vartheta^{-1}
I_{\vartheta^j}/(\vartheta^{-1}I_{\vartheta^j})^p\>. 
\]
Together with part 1) of Proposition~\ref{PropASKgen} and
Proposition~\ref{I}, this proves (\ref{Ome=p}).

Since $0<v\vartheta^{-1}\notin vK$, we have $v\vartheta^{-1}>H_\cE$ and therefore, 
$\vartheta^{-1}\in\cM_\cE\,$. By part 3) of Lemma~\ref{aM^n} it follows that
$(\vartheta^{-1}\cM_\cE)^p\subsetneq\vartheta^{-1}\cM_\cE\,$, which shows 
that $\Omega_{\cO_L|\cO_K}\ne 0$.
%
\end{proof}

With the notation of the statement and proof of Theorem \ref{OASe}, we have that for $z\in
\cO_L$, $z=g( c_{\alpha} \vartheta^j)$ for some $g(X)\in \cO_K[X]$, where $c_{\alpha}\in K$ 
is such that $vc_{\alpha}\vartheta^j>0$ and the universal derivation $d_{\cO_L|\cO_K}$ is 
defined by 
\[
d_{\cO_L|\cO_K}(z)\>=\>[c_{\alpha}g'(\tilde c_{\alpha}\vartheta^j)]\in I_\cE\cM_\cE/
(I_\cE\cM_\cE)^p
\]
by equation (\ref{Compd}).

\sn
%
%
\subsection{Kummer extensions $(L|K,v)$ of degree $p=\chara Kv$ with $\rmf(L|K)=p$}
\mbox{ }\sn
\begin{theorem}                                    \label{OKumf}
Let $(L|K,v)$ be a Kummer extension of degree $p=\rmf(L|K)=\chara Kv$. 
\pars
In case i) of Lemma~\ref{duKpf},
\begin{equation}                     \label{OKumfdisp1}
\Omega_{\cO_L|\cO_K}\>\cong\> \cO_L/(p)\>\cong\>I_\cE/I_\cE^p
\end{equation}
as $\cO_L$-modules. Consequently, $\Omega_{\cO_L|\cO_K}\ne 0$.

In case ii) of Lemma~\ref{duKpf},
\begin{equation}                     \label{OKumfdisp2}
\Omega_{\cO_L|\cO_K}\>\cong\> \cO_L/(p\tilde c^{p-1}) \>\cong\>I_\cE/I_\cE^p
\end{equation}
as $\cO_L$-modules, and
\begin{equation}                 \label{Ome=0}
\Omega_{\cO_L|\cO_K}=0\mbox{ if and only if $Lv|Kv$ is separable.}
\end{equation}
\end{theorem}
\begin{proof}
The first isomorphisms in (\ref{OKumfdisp1})
and (\ref{OKumfdisp2}) follow from Lemma~\ref{duKpf} and (\ref{Kahlercomp}).

In case i),
\[
\cO_L/(p)\>=\>\cO_L/((\zeta_p-1)^{p-1})\>\cong\> (\zeta_p-1)/(\zeta_p-1)^p
\>=\>I_\cE/I_\cE^p
\]
by part 2)a) of Proposition~\ref{PropASKgen}.

\pars
In case ii), where $v\tilde{c}=-v(\eta-1)$,
\[
\cO_L/(p\tilde c^{p-1})\>=\>\cO_L/(\tilde c (\zeta_p-1))^{p-1}\>\cong\> (\tilde c
(\zeta_p-1))/(\tilde c (\zeta_p-1))^p \>=\>I_\cE/I_\cE^p
\]
by part 2)b) of Proposition~\ref{PropASKgen}.

By Lemma~\ref{v(eta-1)}, $Lv|Kv$ is separable if and only if $-v\tilde{c}=v(\eta-1)=
\frac{vp}{p-1}$, i.e., $vp\tilde c^{p-1}=0$. This is equivalent to $(p\tilde c^{p-1})
=\cO_L\,$, and thus to $\Omega_{\cO_L|\cO_K}=0$.
%
\end{proof}

Let the notation be as in  the statement and proof of Theorem \ref{OKumf}. In case i), we 
have that for $z\in \cO_L$,  $z=g(\eta)$ 
for some $g(X)\in \cO_K[X]$ and the universal derivation $d_{\cO_L|\cO_K}$ is defined by 
\[
d_{\mathcal O_L|\cO_K}(z)\>=\>[g'(\eta)]\in\cO_L/(p)
\]
by equation (\ref{Simpd}). 
  
In case ii), we have that for $z\in \cO_L$,  $z=g(\tilde c(\eta-1))$ 
for some $g(X)\in \cO_K[X]$ and the universal derivation $d_{\cO_L|\cO_K}$ is defined by 
\[
d_{\mathcal O_L|\cO_K}(z)\>=\>[g'(\tilde c(\eta-1))]\in\cO_L/(p\tilde c)^{p-1}
\]
by equation (\ref{Simpd}).

\sn
%
%
\subsection{Kummer extensions $(L|K,v)$ of prime degree $q$ with
$\rme(L|K)=q$}                                      \label{sectOKume}
\mbox{ }\sn
\begin{theorem}                                      \label{OKume}
Let $\cE=(L|K,v)$ be a Kummer extension of prime degree $q$ with $\rme(L|K)=q$.
\sn
In case i) of Lemma~\ref{duKpe},
\begin{equation}                                     \label{OKumec}
\Omega_{\cO_L|\cO_K}\>\cong\> \cM_\cE/q\cM_\cE^q
\end{equation}
as $\cO_L$-modules. If $q\ne\chara Kv$, then
\begin{equation}                                     \label{OKumec1}
\cM_\cE/q\cM_\cE^q\>=\> \cM_\cE/\cM_\cE^q \>.
\end{equation}
If $q=\chara Kv$, then
\begin{equation}                                     \label{OKumecp} 
\cM_\cE/q\cM_\cE^q \>\cong\> (\zeta_q-1)\cM_\cE/((\zeta_q-1)\cM_\cE)^q
\>=\> I_\cE\cM_\cE/(I_\cE\cM_\cE)^q\>.
\end{equation}

We have that $\Omega_{\cO_L|\cO_K}= 0$ if and only if $q\notin\cM_\cE$ and $\cM_\cE$ is 
a nonprincipal $\cO_\cE$-ideal. The condition $q\notin\cM_\cE$ always holds when $q\ne 
\chara Kv$.


\mn
In case ii) of Lemma~\ref{duKpe},
\begin{equation}                                     \label{OKumec2} 
\Omega_{\cO_L|\cO_K}\>\cong\> \xi^{-1}\cM_\cE/(\xi^{-1}\cM_\cE)^q
\>=\> I_\cE\cM_\cE/(I_\cE\cM_\cE)^q
\end{equation}
as $\cO_L$-modules; in particular, $\Omega_{\cO_L|\cO_K}\ne 0$.
%
\end{theorem}
\begin{proof}
Assume that case i) holds. 
We will apply Proposition~\ref{LimProp} with $A=\cO_K$ and $B=\cO_L\,$. We set $S=\{\alpha
\in vK\mid \alpha +v\eta>0\}$, endowed with the reverse ordering of $vK$. For each
$\alpha\in S$ we choose $c_\alpha\in K$ such that $vc_\alpha=\alpha$. We
set $b_{\alpha}=c_\alpha\eta$, $a_{\alpha}=c_\alpha$, and $c_{\alpha,\beta}=0$.
Then $a_\beta | a_{\alpha}$ and $A[b_\alpha]\subseteq A[b_\beta]$ if $\alpha\leq\beta$,
and we have that $c_1\eta=\frac{c_1}{c_2}c_2\eta$.
We denote by $h_{\alpha}$ the minimal polynomial of $b_{\alpha}=c_\alpha\eta$
over $K$. Thus in the notation of Lemma~\ref{duKpe}, $h_{\alpha}=f_{c_\alpha}$ so that
$h_{\alpha}'(b_{\alpha})=f_{c_\alpha}'(c_\alpha\eta)$ and $V=\cD_0(\cO_L|\cO_K)=
qI_\eta^{q-1}$ by equation (\ref{fc'etaI}) of Lemma \ref{duKpe}. Further,
$U=(a_\alpha\mid\alpha\in S)=(c_\alpha\mid\alpha\in S)=\eta^{-1}I_\eta\,$. Hence by
Proposition~\ref{LimProp},
\[
\Omega_{\cO_L|\cO_K}\cong U/UV\>\cong\> \eta^{-1}I_\eta/\eta^{-1}I_\eta
\,qI_\eta^{q-1}\>\cong\> I_\eta/qI_\eta^q\>. 
\]
From Proposition~\ref{I} we know that $I_\eta=\cM_\cE\,$. This proves (\ref{OKumec}).

\pars
Assume that $q\ne\chara Kv$. Then $vq=0$, hence $q\notin\cM_\cE$ and  by part 1) of
Lemma~\ref{aM^n}, $q\cM_\cE^q=(q\cM_\cE)\cM_\cE^{q-1}=\cM_\cE\cM_\cE^{q-1}=\cM_\cE^q$.
This proves (\ref{OKumec1}).

\pars
Assume that $q=\chara Kv$. Then 
\[
\cM_\cE/q\cM_\cE^q \>\cong\>(\zeta_q-1)\cM_\cE/(\zeta_q-1)q\cM_\cE^q\>=\>(\zeta_q-1)
\cM_\cE/(\zeta_q-1)^q\cM_\cE^q \>=\> I_\cE\cM_\cE/I_\cE^q\cM_\cE^q\>,
\]
where we have used that $vq=(q-1)v(\zeta_q-1)$ and that $I_\cE=(\zeta_q-1)$ by
part 2)a) of Proposition~\ref{PropASKgen}. This proves (\ref{OKumecp}).

\pars
Now we determine when $\Omega_{\cO_L|\cO_K}= 0$ holds in the present case. If $q\in
\cM_\cE$, then $q\cM_\cE^q\subseteq q\cM_\cE\subsetneq \cM_\cE$ by part 1) of 
Lemma~\ref{aM^n}, and if $\cM_\cE$ is a principal $\cO_\cE$-ideal, then $q\cM_\cE^q
\subseteq \cM_\cE^q\subsetneq \cM_\cE$ by part 2) of Lemma~\ref{aM^n}; hence in both cases,
$\Omega_{\cO_L|\cO_K}\ne 0$. On the other hand, if $q\notin\cM_\cE$ and $\cM_\cE$ is a
nonprincipal $\cO_\cE$-ideal, then by parts 1) and 2) of Lemma~\ref{aM^n},
$q\cM_\cE^q=q\cM_\cE=\cM_\cE\,$, whence $\Omega_{\cO_L|\cO_K}= 0$. 
If $q\ne \chara Kv$, then $vq=0$, hence $q\notin\cM_\cE\,$.

\parm
Assume that case ii) holds. 
Again we will apply Proposition~\ref{LimProp} with $A=\cO_K$ and $B=\cO_L\,$. We set
$S=\{\alpha\in vK\mid \alpha+v\xi^j>0\}$, endowed with the reverse ordering of $vK$.
For each $\alpha\in S$ we choose $c_\alpha\in K$ such that $vc_\alpha=\alpha$. We
set $b_{\alpha}=c_\alpha\xi^j$, $a_{\alpha}=c_\alpha$, and $c_{\alpha,\beta}=0$.
Then $a_\beta | a_{\alpha}$ and $A[b_\alpha]\subseteq A[b_\beta]$ if $\alpha\leq\beta$; 
we have that $c_1\xi^j=\frac{c_1}{c_2}c_2\xi^j$.
We denote by $h_{\alpha}$ the minimal polynomial of $b_{\alpha}=c_\alpha\xi^j$ over $K$.
Thus in the notation of Lemma~\ref{duKpe}, $h_{\alpha}=h_{j,\alpha}$ so that
$h_{\alpha}'(b_{\alpha})=h_{j,\alpha}'(c_\alpha\xi^j)$ and $V=\cD_0(\cO_L|\cO_K)=
(\xi^{-1}I_{\xi^j})^{q-1}$ by equation (\ref{hc'xiI}) of Lemma \ref{duKpe}. 
Further, $U=(a_\alpha\mid\alpha\in S)=(c_\alpha\mid\alpha\in S)=\xi^{-j}I_{\xi^j}\,$. 
Hence by Proposition~\ref{LimProp}, $\Omega_{\cO_L|\cO_K}\cong U/UV\cong \xi^{-j}I_{\xi^j}
/\xi^{-j}I_{\xi^j}\,(\xi^{-1}I_{\xi^j})^{q-1}\cong \xi^{-1}I_{\xi^j}/\xi^{-1}I_{\xi^j}\,
(\xi^{-1}I_{\xi^j})^{q-1}=\xi^{-1}I_{\xi^j}/(\xi^{-1} I_{\xi^j})^q$. Again from
Proposition~\ref{I} we know that $I_{\xi^j}=\cM_\cE\,$. Further, $(\xi^{-1})=I_\cE$ by 
part 2)b) of Proposition~\ref{PropASKgen}. This proves (\ref{OKumec2}).

Since $0<v\xi^{-1}\notin vK$, we have $v\xi^{-1}>H_\cE$ and therefore, $\xi^{-1}\in
\cM_\cE\,$. By part 3) of Lemma~\ref{aM^n} it follows that $(\xi^{-1}\cM_\cE)^q
\subsetneq\xi^{-1}\cM_\cE\,$, which shows that $\Omega_{\cO_L|\cO_K}\ne 0$.
%
\end{proof}

Let the notation be as in the statement and proof of Theorem \ref{OKume}. In case i), we have
that for $z\in \cO_L$,  $z=g(c_{\alpha}\eta)$ for some $c_{\alpha}\in K$ such that 
$vc_{\alpha}\vartheta>0$ and $g(X)\in \cO_K[X]$ and the universal derivation 
$d_{\cO_L|\cO_K}$ is defined by 
\[
d_{\cO_L|\cO_K}(z)\>=\>[c_{\alpha}g'(c_{\alpha}\eta)]\in \cM_\cE/q\cM_\cE^q
\]
by equation (\ref{Compd}). 
  
In case ii), we have that for $z\in \cO_L$,  $z=g(c_{\alpha} \xi^j)$ 
for some $c_{\alpha}\in K$ such that $vc_{\alpha}\xi^j>0$ and $g(X)\in \cO_K[X]$ and the
universal derivation $d_{\cO_L|\cO_K}$ is defined by 
\[
d_{\mathcal O_L|\cO_K}(z)\>=\>[c_{\alpha}g'(c_{\alpha}\xi^j)]\in  
I_\cE\cM_\cE/(I_\cE\cM_\cE)^p
\]
by equation (\ref{Compd}).

\bn
%
%
\section{K\"ahler differentials of towers of Galois extensions}     \label{secKahExact}
In this section, our goal is the proof of the following two theorems, which will be given
in Subsection~\ref{sectpfKahler}. We begin by  preparing the
ingredients for the proofs.

We first state the ``first fundamental exact sequence'' of K\"ahler differentials.
\begin{theorem}(\cite[Theorem 25.1]{Mat})                  \label{FES}
A composite $k\rightarrow A\rightarrow B$ of ring homomorphisms leads to a natural exact 
sequence of $B$-modules
$$
\Omega_{A|k}\otimes_AB\rightarrow \Omega_{B|k}\rightarrow\Omega_{B|A}\rightarrow 0.
$$ 
\end{theorem}

We will verify that in relevant situations, the left most homomorphism is injective, giving 
a short exact sequence. The following theorem is a consequence of the more general 
Theorem~6.3.32 of \cite{GR}. However, we will give an alternate proof in
Section~\ref{sectpfKahler}.

\begin{theorem}\label{GalTow}
Assume that $L|K$ and $M|L$ are towers of finite Galois extensions of valued fields. 
Then there is a natural short exact sequence
$$
0\rightarrow \Omega_{\cO_L|\cO_K}\otimes_{\cO_L}\cO_M\rightarrow \Omega_{\cO_M|\cO_K}
\rightarrow \Omega_{\cO_M|\cO_L}\rightarrow 0.
$$
In particular, $\Omega_{\cO_M|\cO_K}=0$ if and only if $\Omega_{\cO_M|\cO_L}=0$ and 
$\Omega_{\cO_L|\cO_K}=0$.
\end{theorem}

\begin{theorem}                                       \label{Kahler}
Let $(K,v)$ be a valued field. Then
\begin{enumerate}
\item[1)] $\Omega_{\cO_{K\sep}|\cO_K}=0$ if and only if $\Omega_{\cO_L|\cO_K}=0$
for all finite Galois subextensions $L|K$ of $K\sep$.
\item[2)] Let $L|K$ be a finite Galois subextension of $K\sep$ and assume that 
$$
K\subset K\ina=K_0\subset K_1\subset \cdots\subset K_\ell=L
$$
is a tower of field extensions factoring $L|K$ such that 
$K\ina$ is the inertia field of $(L|K,v)$ and $K_{i+1}|K_i$ is Galois of prime 
degree for all $i$. Then $\Omega_{\cO_L|\cO_K}=0$ if and only if 
$\Omega_{\cO_{K_{i+1}}|\cO_{K_i}}=0$ for $0\le i\le \ell-1$.
\end{enumerate}
\end{theorem}

\begin{lemma}                          \label{KahCalc36} 
Assume that $(L|K,v)$  is a valued field extension. Then 
$\cO_{L}$ is a faithfully flat $\cO_K$-module. 
\end{lemma}

\begin{proof}
We have that  $\cO_{L}$ is a flat $\cO_{K}$-module by \cite[Theorem 4.33]{Rot}
(see also \cite[Theorem 4.35]{Rot2}), since 
$\cO_{K}$  is a valuation ring and $\cO_{L}$ is a torsion free $\cO_K$-module. Further, 
$\cO_{L}$ is a faithfully flat $\cO_{K}$-module by Theorem 7.2 \cite{Mat}, since 
$\cM_{K}\cO_{L}\ne \cO_{L}$.
\end{proof}



\begin{lemma} \label{KahCalc9} 
Let $(L|K,v)$ be a finite valued field extension which is unibranched and such that 
there is a tower of field extensions $K=K_0\subset K_1\subset \cdots \subset K_\ell=L$ 
such that for $1\le i\le \ell$ one of the following holds:
\begin{enumerate}
\item[1)]  $K_i|K_{i-1} $ is Galois of prime degree or 
\item[2)]  $[K_i:K_{i-1}]=[K_iv:K_{i-1}v]$ and $K_iv$ is separable over $K_{i-1}v$.
\end{enumerate}
Then for $2\le i\le \ell$, we have natural short exact sequences
 \begin{equation}\label{KahCalc8}
 0\rightarrow (\Omega_{\cO_{K_{i-1}}|\cO_K})\otimes_{\cO_{K_{i-1}}}\cO_{K_i}
 \rightarrow \Omega_{\cO_{K_i}|\cO_K}\rightarrow \Omega_{\cO_{K_i}|\cO_{K_{i-1}}}
 \rightarrow 0.
 \end{equation} 
\end{lemma}
%
%
\begin{proof}  
By Theorem \ref{seprfe}, Theorem \ref{thmgen} for unibranched defectless extensions 
of prime degree and \cite[Lemma 2.3,  Lemma 3.1, Lemma 3.2 and Proposition 3.3]{CuKuRz} 
for extensions of prime degree with nontrivial defect
for $1\le i\le \ell$ there exist directed sets $S_i$ 
with associated $\alpha(i)_j\in K_i$ for $j\in S_i$ such that 
$\cO_{K_{i-1}}[\alpha(i)_j]\subset \cO_{K_{i-1}}[\alpha(i)_k]$ if $j\le k$ and
$\cO_{K_i}=\cup_{j\in S_i}\cO_{K_{i-1}}[\alpha(i)_j]$. Further, $\cO_{K_i}[\alpha(i)_j]
\cong \cO_{K_i}[X]/(f_i^j(X))$ where $f_i^j(X)$ is the minimal polynomial $\alpha(i)_j$ 
over $K_{i-1}$.

Let $T_i$ be the set of $(k_1,k_2,\ldots,k_{i-1},k_i)\in S_1\times S_2\times\cdots\times
S_i$ such that 
$f_n^{k_n}(x)\in \cO_K[\alpha(1)_{k_1},\alpha(2)_{k_2},\ldots,\alpha(n-1)_{k_{n-1}}][x]$
for $2\le n\le i$. We define a partial order on $T_i$ by the rule $(k_1,\ldots,k_i)\le 
(l_1,\ldots,l_i)$ if $k_m\le l_m$ for $1\le m\le i$. The $T_i$ are directed sets since 
the $S_i$ are, and setting 
\[
A_{k_1,\ldots,k_i}\>=\>\cO_K[\alpha(1)_{k_1},\alpha(2)_{k_2},\ldots,
\alpha(i-1)_{k_{i-1}},\alpha(i)_{k_i}]
\]
for $(k_1,\ldots,k_i)\in T_i$, we have inclusions 
\[
A_{k_1,\ldots,k_i}\subset A_{l_1,
\ldots,l_i} \mbox{ for } (k_1,\ldots,k_i)\le (l_1,l_2,\ldots,l_i) \mbox{ in $T_i$.}
\]

By our construction, for $2\le m\le i$, there exist 
$$
g_m^{k_m}(X_1,\ldots, X_{m-1},X_m)\in \cO_K[X_1,X_2,\ldots,X_{m-1},X_m]
$$
such that $g_m^{k_m}(\alpha(1)_{k_1},\cdots ,\alpha(m-1)_{k_{m-1}},X_m)=f_m^{k_m}(X_m)$.
 
By \cite[Theorem 1]{HTT}, we have that
$$
A_{k_1,\ldots,k_i}\cong \cO_K[X_1,\ldots,X_i]/(g_1^{k_1}(X_1), g_2^{k_2}(X_1,X_2),
\ldots,g_i^{k_i}(X_1,\ldots,X_i)).
$$

By \cite[Theorem 25.2]{Mat}, $\Omega_{A_{k_1,\ldots,k_i}|\cO_K}\cong
(A_{k_1,\ldots,k_i}dX_1\oplus\cdots\oplus A_{k_1,\ldots,k_i}dX_i)/U_{k_1,\ldots,k_i}$, 
where $U_{k_1,\ldots,k_i}$ is the $A_{k_1,\ldots,k_i}$-submodule of 
 $A_{k_1,\ldots,k_i}dX_1\oplus\cdots\oplus A_{k_1,\ldots,k_i}dX_i$ generated  by
 \begin{equation}\label{KahCalc1}
  \left[\frac{\partial f_1^{k_1}}{\partial X_1}(\alpha(1)_{k_1})\right]dX_1
  \end{equation}
and
 \begin{equation}\label{KahCalc2} 
 \left[\frac{\partial g_m^{k_m}}{\partial X_1}(\alpha(1)_{k_1},\ldots,\alpha(m)_{k_m}
 )\right]dX_1 
 +\cdots+\left[\frac{\partial g_m^{k_m}}{\partial X_m}(\alpha(1)_{k_1},\ldots,
 \alpha(m)_{k_m})\right]dX_m
  \end{equation}
for $2\le m\le i$. We further have that 
 \begin{equation}\label{KahCalc3}
 \left[\frac{\partial f_1^{k_1}}{\partial X_1}(\alpha(1)_{k_1})
 \right]=(f_1^{k_1})'(\alpha(1)_{k_1})
 \end{equation}
and 
 \begin{equation}\label{KahCalc4}
 \left[\frac{\partial g_m^{k_m}}{\partial X_m}(\alpha(1)_{k_1},\ldots,\alpha(m)_{k_m})
 \right]=(f_m^{k_m})'(\alpha(m)_{k_m})
 \end{equation}
for $2\le m\le i$.

By Theorem \ref{FES}, we have a natural exact sequence of $A_{k_1,\ldots,
k_{i}}$-modules
  \begin{equation}\label{KahCalc10}
 \Omega_{A_{k_1,\ldots,k_{i-1}}|\cO_K}\otimes_{A_{k_1,\ldots,k_{i-1}}} A_{k_1,\ldots,k_i}
 \rightarrow
  \Omega_{A_{k_1,\ldots,k_{i}}|\cO_K}\rightarrow
  \Omega_{A_{k_1,\ldots,k_{i}}|  A_{k_1,\ldots,k_{i-1}}}\rightarrow 0.
   \end{equation}
   
For $(k_1,\ldots,k_i)\in T_i$, let 
 $$
 L_{k_1,\ldots,k_i}=\Omega_{A_{k_1,\ldots,k_{i-1}}|\cO_K}\otimes_{A_{k_1,\ldots,k_{i-1}}}
 \cO_{K_i},
 $$
  $$
  M_{k_1,\ldots,k_i}=\Omega_{A_{k_1,\ldots,k_{i}}|\cO_K}\otimes_{A_{k_1,\ldots,k_{i}}}
  \cO_{K_i},
  $$
  $$
  N_{k_1,\ldots,k_i}= \Omega_{A_{k_1,\ldots,k_{i}}|  A_{k_1,\ldots,k_{i-1}}}
  \otimes_{A_{k_1,\ldots,k_{i}}} \cO_{K_i}.
  $$

Applying the right exact functor $\otimes_{A_{k_1,\ldots,k_{i}}}\cO_{K_i}$ to
(\ref{KahCalc10}), we have an exact sequence of  $\cO_{K_i}$-modules
  \begin{equation}\label{KahCalc5}
  L_{k_1,\ldots,k_i}\stackrel{u}{\rightarrow} M_{k_1,\ldots,k_i}\rightarrow N_{k_1,
  \ldots,k_i}\rightarrow 0.
     \end{equation}

Now $\Omega_{A_{k_1,\ldots,k_{i-1}}|\cO_K}\otimes_{A_{k_1,\ldots,k_{i-1}}} \cO_{K_i}$ 
is the quotient of $\cO_{K_i}dX_1\oplus\cdots \oplus\, \cO_{K_i}dX_{i-1}$ by the
relations (\ref{KahCalc1}) and (\ref{KahCalc2}) for $2\le m\le i-1$ and 
$\Omega_{A_{k_1,\ldots,k_{i}}|\cO_{K}}\otimes_{A_{k_1,\ldots,k_{i}}} \cO_{K_i}$ is the
quotient of $\cO_{K_i}dX_1\oplus\cdots \oplus \,\cO_{K_i}dX_{i}$ by the relations
(\ref{KahCalc1}) and (\ref{KahCalc2}) for $2\le m\le i$. Since 
$(f_i^{k_i})'(\alpha(i)_{k_i})\ne 0$ (as $K_i$ is separable over $K_{i-1}$) we have by
(\ref{KahCalc4}) with $m=i$ that $u$ is injective,
so that (\ref{KahCalc5}) is actually short exact.

Let $(k_1,\ldots,k_i)$ and $(l_1,\ldots,l_i)$ in $T_i$ be such that $(k_1,\ldots,k_i)
\le (l_1,\ldots,l_i)$.  Then we have a natural commutative diagram  of 
$\cO_{K_i}$-modules with short exact rows
 \begin{equation}\label{KahCalc6}
 \begin{array}{ccccccccc}
 0&\rightarrow &L_{k_1,\ldots,k_i}&\rightarrow&
 M_{k_1,\ldots,k_i}&\rightarrow&
 N_{k_1,\ldots,k_i} &   \rightarrow& 0\\
 &&\downarrow&&\downarrow&&\downarrow\\
 0&\rightarrow &L_{l_1,\ldots,l_i}&\rightarrow&
 M_{l_1,\ldots,l_i}&\rightarrow&
 N_{l_1,\ldots,l_i} &   \rightarrow& 0  \end{array}
   \end{equation}
where  the vertical arrows are the natural maps determined by the differentials of the
inclusions of 
 $A_{k_1,\ldots,k_{i-1}}$ into $A_{l_1,\ldots,l_{i-1}}$ and of $A_{k_1,\ldots,k_{i}}$ 
 into $A_{l_1,\ldots,l_i}$.
By \cite[Theorem 2.18]{Rot} (see also \cite[Proposition 5.33]{Rot2}), we have a short 
exact sequence of $\cO_K$-modules
\begin{equation}\label{KahCalc7}
0\rightarrow \lim_{\rightarrow}L_{k_1,\ldots,k_i}\rightarrow
 \lim_{\rightarrow}M_{k_1,\ldots,k_i}\rightarrow
 \lim_{\rightarrow}N_{k_1,\ldots,k_i}    \rightarrow 0.
 \end{equation}
 
 By our construction of $T_i$, we have that $\cup A_{k_1,\ldots,k_i}=\cO_{K_i}$, where the
 union is over all $(k_1,\ldots,k_i)\in T_i$. Thus $\displaystyle\lim_{\rightarrow}
 M_{k_1,\ldots,k_i}
 \cong \Omega_{\cO_{K_i}|\cO_K}$ by \cite[Theorem 16.8]{Eis}.
 We also have that $\cup A_{k_1,\ldots,k_{i-1}}=\cO_{K_{i-1}}$, where the union is over 
 all $(k_1,\ldots,k_{i-1})$ such that $(k_1,\ldots,k_i)\in T_i$.  Thus 
 $$
 \lim_{\rightarrow}\left(\Omega_{A_{k_1,\ldots,k_{i-1}}|\cO_K}\otimes_{A_{k_1,
 \ldots,k_{i-1}}}\cO_{K_{i-1}}\right) 
 \cong \Omega_{\cO_{K_{i-1}}|\cO_K}
 $$
 again by \cite[Theorem 16.8]{Eis}. Now 
 $$
 \begin{array}{lll}
 \lim_{\rightarrow} L_{k_1,\ldots,k_i}&=&
 \lim_{\rightarrow}\left( \Omega_{A_{k_1,\ldots,k_{i-1}}|\cO_K}\otimes_{A_{k_1,
 \ldots,k_{i-1}}} \cO_{K_i}\right)\\
 &\cong&
 \lim_{\rightarrow}\left((\Omega_{A_{k_1,\ldots,k_{i-1}}|\cO_K}\otimes_{A_{k_1,
 \ldots,k_{i-1}}}\cO_{K_{i-1}})\otimes_{\cO_{K_{i-1}}} \cO_{K_i}\right)\\
 &\cong& \left(\lim_{\rightarrow}(\Omega_{A_{k_1,\ldots,k_{i-1}}|\cO_K}\otimes_{A_{k_1,
 \ldots,k_{i-1}}}\cO_{K_{i-1}})\right)\otimes_{\cO_{K_{i-1}}} \cO_{K_i}\\

 &\cong&\Omega_{\cO_{K_{i-1}}|\cO_K}\otimes_{\cO_{K_{i-1}}}\cO_{K_i}
   \end{array} 
 $$ 
 where the equality of the third row is by \cite[Corollary 2.20]{Rot}.
 
 We have that 
 $$
 A_{k_1,\ldots,k_i}=A_{k_1,\ldots,k_{i-1}}[\alpha(i)_{k_i}]\cong A_{k_1,\ldots,k_{i-1}}
 [X_i]/(f_i^{k_i}),
 $$
 so
 $$
 \Omega_{A_{k_1,\ldots,k_{i}}|  A_{k_1,\ldots,k_{i-1}}}
 \cong \left(A_{k_1,\ldots,k_i}/(f_i^{k_i})'(\alpha(i)_{k_i})\right)dX_i
 $$
  by equation (\ref{Kahlercomp}). Since $f_i^{k_i}$ is the minimal polynomial of 
  $\alpha(i)_{k_i}$ over $K_{i-1}$, we have that 
  $$
  \Omega_{\cO_{K_{i-1}}[\alpha(i)_{k_i}]|\cO_{K_{i-1}}}\cong
  \cO_{K_{i-1}}[\alpha(i)_{k_i}]/((f_i^{k_i})'(\alpha(i)_{k_i}))dX_i
  $$
  also by  (\ref{Kahlercomp}).
  Thus 
 $$
 \begin{array}{lll}
 N_{k_1,\ldots,k_i}&=&\Omega_{A_{k_1,\ldots,k_{i}}|  A_{k_1,\ldots,k_{i-1}}}
 \otimes_{A_{k_1,\ldots,k_{i}}} \cO_{K_i}
 \cong \left(\cO_{K_i}/(f_i^{k_i})'(\alpha(i)_{k_i})\right)dX_i\\
 &\cong&   \left(\Omega_{\cO_{K_{i-1}}[\alpha(i)_{k_i}]|\cO_{K_{i-1}}}\right)
 \otimes_{\cO_{K_{i-1}[\alpha(i)_{k_i}]}}\cO_{K_i}.
 \end{array}
  $$ 
  
Since $\cup \cO_{K_{i-1}}[\alpha(i)_{k_i}]=\cO_{K_i}$,  we have that
$\displaystyle\lim_{\rightarrow}N_{k_1,\ldots,k_i}\cong \Omega_{\cO_{K_i}|\cO_{K_{i-1}}}$ 
by \cite[Theorem 16.8]{Eis}.
  
In conclusion, for $1\le i\le r$, the sequence (\ref{KahCalc8}) is isomorphic to the 
short exact sequence (\ref{KahCalc7}).
\end{proof}

In Definition 2 of Chapter I, page 11 \cite{Ra}, an \'etale algebra is defined. Let $A$ be 
a ring and $B$ be an $A$-algebra. $B$ is \'etale over $A$ if
\begin{enumerate}
\item[1)] $B$ is an $A$-algebra of finite presentation and
\item[2)] For all $A$-algebra $D$ and ideals $J$ of $D$ such that $J^2=0$, the natural map 
$\mbox{Hom}_{A\mbox{-alg}}(B,D)\rightarrow \mbox{Hom}_{A\mbox{-alg}}(B,D/J)$ is a bijection.
\end{enumerate}
In Definition IV.17.3.1 \cite{EGA}, an \'etale morphism of schemes is defined. After the
definition, it is shown that a morphism of affine schemes $\mbox{Spec}(B)\rightarrow 
\mbox{Spec}(A)$ is \'etale if and only if $B$ is  \'etale over $A$.

\begin{proposition}                         \label{KahCalc30} 
Let $(L|K,v)$ be a finite Galois extension of valued fields. Let $G$ be the Galois group 
of $L|K$ and let $H$ be a subgroup of $G$ which contains the inertia group of $L|K$. 
Denote the fixed field of $H$ in $L$ by $L_0\,$. Then $\Omega_{\cO_{L_0}|\cO_K}=0$.
\end{proposition}

\begin{proof}  
Let $A=\cO_K$, $C$ be the integral closure of $A$ in $\cO_L$ and $B=C^H$ be the integral 
closure of $A$ in $L_0=L^H$. There exists a maximal ideal $r$ of $C$ such that $C_r=\cO_L$. 
Let $n=r\cap B$, the maximal ideal of $B$, so that $\cO_{L_0}=B_n$. By Theorem 1 of 
Chapter X, page 103 \cite{Ra}, there exists $f\in B\setminus n$ such that $B'=B_f$ is an 
\'etale $A$-algebra. We have that $(B')_{n_f}=\cO_{L_0}$. $\mbox{Spec}(B')\rightarrow 
\mbox{Spec}(A)$ is an \'etale morphism, so the map is formally unramified (Definition 
IV.17.1.1 \cite{EGA}). Thus $\Omega_{B'|A}=0$ by Proposition IV.17.2.1 \cite{EGA}. Thus 
$0=(\Omega_{B'|A})\otimes_{B'}(B')_{n_f}=\Omega_{\cO_{L_0}|\cO_K}$ by 
\cite[Proposition 16.9]{Eis}.
\end{proof}

\begin{proposition}                        \label{KahCalc20*} 
Assume that $(L|K,v)$ is a finite Galois extension. Then
 $$
 \Omega_{\cO_{L}|\cO_K}\cong \Omega_{\cO_L|\cO_{K\ina}}
 $$
where $K\ina$ is the inertia field of $(L|K,v)$.
\end{proposition}

\begin{proof} This follows from Proposition \ref{KahCalc30} and the exact 
sequence of Theorem \ref{FES}.
\end{proof}

\sn

We now give the proof of Theorem \ref{ThmI1}.
Let $p$ be the characteristic of the residue field $Kv$ 
and $q=[L:K]$ a prime number. The description of $\Omega_{\cO_L|\cO_K}$ 
and the characterization of vanishing of this module depend, among other information, 
on the invariants of the valued field extension that appear in the following product: 
$$
q\>=\>[L:K]\>=\>\rmd(L|K)\,\rme(L|K)\,\rmf(L|K)\,\rmg(L|K)
$$
where $\rme(L|K)=(vL:vK)$, $\rmf(L|K)=[Lv:Kv]$, $\rmg(L|K)$ is the number of distinct
extensions of $v|K$ to $L$ and $\rmd(L|K)$ is the defect of the extension, which is a 
power of $p$. Since $q$ is a prime, exactly one of the factors will be equal to $q$, 
and the others equal to $1$.
The description of $\Omega_{\cO_L|\cO_K}$ also depends on the rank and the structure of 
the value group of $(K,v)$ if $\rmd(L|K)\ne 1$ or $\rme(L|K)\ne 1$,
and on whether $Lv|Kv$ is separable or inseparable if $\rmf(L|K)\ne 1$.

In the case of $\rmd(L|K)=p$, our results are proven in \cite[Theorem 1.2]{CuKuRz}. 
In the case of $\rme(L|K)=q$, they are obtained in Theorem \ref{OASe} for Artin-Schreier
extensions and Theorem \ref{OKume} for Kummer extensions. If $\rmf(L|K)=q$, then they 
are obtained in  Theorems \ref{seprfe} and \ref{OASf} for Artin-Schreier extensions and
Theorems \ref{seprfe} and \ref{OKumf} for Kummer extensions. 

In the remaining case when $\rmg(L|K)=q$, the extension $(L|K,v)$ is an inertial 
extension. Thus $\Omega_{\cO_L|\cO_K}=0$ by Proposition \ref{KahCalc30}.

\sn
%
%
\subsection{Henselization}                      \label{secthens}
\mbox{ }\sn
We now recall some facts about henselization of fields and rings. 
A valued field $(K,v)$ is \bfind{henselian} if it satisfies Hensel's Lemma, or 
equivalently, all of its algebraic extensions are unibranched
(cf.\ \cite[Section 16]{En}). 

An extension $(K^h,v^h)$ of a valued field $(K,v)$ is called a \bfind{henselization} 
of $(K,v)$ 
if $(K^h,v^h)$ is henselian and for all henselian valued fields $(L,\omega)$ and all 
embeddings $\lambda:(K,v)\rightarrow (L,\omega)$, there exists a unique embedding 
$\tilde\lambda:(K^h,v^h)\rightarrow (L,\omega)$ which extends $\lambda$.

A henselization $(K^h,v^h)$ of $(K,v)$ can be constructed by choosing an extension $v^s$ 
of $v$ to a separable closure $K\sep$ of $K$ and letting $K^h$ be the fixed field of the
decomposition group  
$$
G^d(K\sep|K)=\{\sigma\in G(K\sep|K)\mid v^s\circ\sigma=v^s\}
$$
of $v^s$, and defining $v^h$ to be the restriction of $v^s$ to $K^h$ 
(\cite[Theorem 17.11]{En}).  The valuation ring $\mathcal O_{K^h}$ of $v^H$ is then 
\begin{equation}\label{eqN1}
\cO_{K^h}=\cO_{v^s}\cap K^h=\tilde A_{\tilde m}
\end{equation}
where $\tilde A$ is the integral closure of $\cO_v$ in $K^h$ and $\tilde m=\mathcal 
M_{K^{\rm sep}}\cap K^h$.

The definition of a henselian local ring is given in Definition 1, Chapter I, page 1 
of \cite{Ra}. A local ring $A$ is henselian if all finite $A$-algebras $B$ are a product 
of local rings. 

Assume that $A$ is a local ring and $g(X)\in A[X]$ is a polynomial. Let $\bar 
g(X)\in A/m_A[X]$ be the polynomial obtained by reducing the coefficients of $g(X)$ 
modulo $m_A$. 

By Proposition 5, Chapter I, page  2 \cite{Ra}, a local ring $A$ is a henselian local ring 
if and only if it has the following property: Let $f(X)\in
A[X]$ be a monic polynomial of degree $n$. If $\alpha(X)$ and $\beta(X)$ are relatively
prime monic polynomials in $A/m_A[X]$ of degrees $r$ and $n-r$ respectively such that 
$\bar f(X)=\alpha(X)\beta(X)$, then there exist monic polynomials $g(X)$ and 
$h(X)$ in $A[X]$ of degrees $r$ and $n-r$ respectively such that $\bar g(X)=
\alpha(X)$, $\bar h(X)=\beta(X)$ and $f(X)=g(X)h(X)$.

Henselization of a local ring is defined in Definition 1, Chapter VIII, page 80 \cite{Ra}.
If $A$ is a local ring, a local ring $A^h$ which dominates $A$ is called a henselization 
of $A$ if any local homomorphism from $A$ to a henselian local ring can be uniquely 
extended to $A^h$.  A henselization always exists, as is shown in 
\cite[Theorem 1, Chapter VIII, page 87]{Ra}. The 
construction is particularly nice when $A$ is a normal local ring, as shown in 
\cite[Theorem 2, Chapter X, page 110]{Ra} (cf.\ \cite[Theorem 43.5]{Na}).
We now explain this construction. Let $K$ be the 
quotient field of $A$ and Let $K\sep$ be a separable closure of $K$. Let $\bar A$ 
be the integral closure of $A$ in $K\sep$ and let $\bar m$ be a maximal ideal of 
$\bar A$.

Let $H$ be the decomposition group
$$
H=G^d(\bar A_{\bar m}|A)=\{\sigma\in G(K\sep|K)\mid \sigma(\bar 
A_{\bar m})=\bar A_{\bar m}\}.
$$
Then
\begin{equation}\label{eqN2}
A^h=(\tilde A)_{\bar m\cap\tilde A}
\end{equation}
where $\tilde A$ is the integral closure of $A$ in $(K^{\rm sep})^H$.

\begin{lemma}                                  \label{Lemma3} 
Assume that $(K,v)$ is a valued field and $(K^h,v^h)$ is a henselization of $K$.
 Then there is a natural isomorphism
$$
\cO_{K^h}\cong \cO_{K}^h.
$$
\end{lemma}

\begin{proof} Let $v^s$ be an extension of $v$ to $K\sep$ and 
$$
H=\{\sigma\in {\rm Gal}(K\sep|K)\mid v^s\circ\sigma=v^s\},
$$
so that $K^h$ is the fixed field of $H$ in $K^{\rm sep}$. Let $\bar V$ be the integral 
closure of $\cO_K$ in $K\sep$,
and let $m=\bar V\cap \cM_{K\sep}$, a maximal ideal in $\bar V$. Since
$K\sep$  is algebraic over $K$, we have that $\cO_{K\sep}=\bar V_m$ by \cite[Theorem
12, page 27]{ZS2}. Now, as is shown on the bottom of page 68 of \cite{ZS2}, $H$ is the
decomposition group
$$
H=G^d(\cO_{K\sep}|\cO_K)=\{\sigma\in G(K\sep|K)\mid \sigma(\cO_{K\sep})=\cO_{K\sep}\},
$$
so that 
$$
\cO_K^h=\cO_{K^h}
$$
by (\ref{eqN1}) and (\ref{eqN2}),
establishing the lemma. 
\end{proof}

\begin{lemma}                            \label{KahCalc18} 
Let $K$ be a valued field and $L$ be a field such that $K\subset L\subset K^h$. Then 
$\Omega_{\cO_{L}|\cO_K}=0$.
\end{lemma}

\begin{proof} 
Let $v^s$ be an extension of $v$ to $K^{\rm sep}$. The field $K^{\rm sep}$ is the directed 
union $K^{\rm sep}=\cup_iM_i$ of the finite Galois extensions $M_i$ of $K$ in $K^{\rm sep}$.  
If $M$ is a finite Galois extension of $K$ in $K^{\rm sep}$, then restriction induces a
surjection of Galois groups $G(K^{\rm sep}|K)\rightarrow G(M|K)$, and an isomorphism 
$G(M|K)\cong G(K^{\rm sep}|K)/G(K^{\rm sep}|M)$. We have an isomorphism of profinite 
groups (\cite[Example 1, page 271]{NE} or \cite[Theorem VI.14.1, page 313]{[L]})
$$
G(K^{\rm sep}|K)\cong \lim_{\leftarrow}G(M_i|K).
$$
Let $G^d(M|K)$ be the decomposition group of the valued field extension $M|K$, for $M$ a
Galois extension of $K$ which is contained in $K^{\rm sep}$ (where the valuation of $M$ 
is $v^s|M$). For $M$ a finite Galois
extension of $K$, restriction induces a homomorphism
\begin{equation}\label{eqDH}
G^d(K^{\rm sep}|K)\rightarrow G^d(M|K).
\end{equation} 
Let $\sigma\in G^d(M|K)$. If $N$ is a finite Galois extension of $M$ contained in 
$K^{\rm sep}$, then there exists $\bar\sigma\in G(N|K)$ such that $\bar\sigma
|_M=\sigma$. Let $A$ be the integral closure of $\cO_M$ in $N$. There exists a maximal 
ideal $p$ of $A$ such that $A_p\cong \cO_N$. Let $q=\bar\sigma(p)$, a maximal ideal 
of $A$. The group $G(N|M)$ acts transitively on the maximal ideals of $A$ 
(\cite[Lemma 21.8]{AG}) so there exists $\tau\in G(N|M)$ such that $\tau(q)=p$. Thus 
$\tau\bar\sigma(\cO_N)=\cO_N$ and $\tau\bar\sigma|_M=\sigma$ and so the 
homomorphism (\ref{eqDH}) is surjective with Kernel $G^d(K^{\rm sep}|K)\cap 
G(K^{\rm sep}|M)$. We have that 
$$
K^h=(K^{\rm sep})^{G^d(K^{\rm sep}|K)}=\cup M_i^{G^d(M_i|K)}.
$$
Thus 
$$
L=L\cap(\cup_iM_i^{G^d(M_i|K)})=\cup L_i
$$
where $L_i=L\cap M_i^{G^d(M_i|K)}$. We have that $\Omega_{\cO_{L_i}|\cO_K}=0$ for all 
$i$ by Proposition \ref{KahCalc30}. Thus 
$$
\Omega_{\cO_L|\cO_K}=\lim_{\rightarrow}(\Omega_{\cO_{L_i}|\cO_K}\otimes_{\cO_{L_i}}\cO_L)
=0
$$
by \cite[Theorem 16.8]{Eis}.
\end{proof}


Let $K$ be a valued field. Fix an extension $v^s$ of $v$ to the separable closure 
$K\sep$ of $K$. The field $K\sep$ is henselian (for instance by the construction before
Lemma \ref{Lemma3}); that is,  the henselization $(K\sep)^h=K\sep$ and $\cO_{(K\sep)^h}
=\cO_{K\sep}$.

\begin{proposition}                                   \label{KahCalc17} 
Let $(K,v)$ be a valued field. Then 
$\Omega_{\cO_{K\sep}|\cO_K}\cong \Omega_{\cO_{K\sep}|\cO_{K^h}}$.
\end{proposition}

\begin{proof} We may embed $K^h$ into  $K\sep$ (by the construction before Lemma
\ref{Lemma3}) giving a tower of valued field extensions $K\subset K^h\subset K\sep$. By
Theorem \ref{FES}, we have an exact sequence $\Omega_{\cO_{K^h}|\cO_K}
\otimes_{\cO_{K^h}}\cO_{K\sep}\rightarrow \Omega_{\cO_{K\sep}|\cO_K}\rightarrow
\Omega_{\cO_{K\sep}|\cO_{K^h}}\rightarrow 0$. The proposition now follows from Lemma
\ref{KahCalc18}.
\end{proof}

\begin{lemma}                               \label{LemmaRNs} 
Assume that  $(L|K,v)$ is a finite separable extension 
of valued fields. Then 
 \begin{equation}                              \label{KahCalc33}
 \Omega_{\cO_L^h|\cO_K^h}\cong (\Omega_{\cO_L|\cO_K})\otimes_{\cO_L}\cO_{L^h}\>.
 \end{equation}
In particular, by Lemma \ref{KahCalc36}, we have that $\Omega_{\cO_L|\cO_K}=0$ if and 
only if $\Omega_{\cO_{L^h}|\cO_{K^h}}=0$.
\end{lemma}
 
\begin{proof}
 We have that 
 \begin{equation}\label{eqRN1}
 \Omega_{\cO_{L^h}|\cO_{K^h}}\cong \Omega_{\cO_{L^h}|\cO_K}
 \end{equation}
by Lemma \ref{KahCalc18} and the exact sequence of Theorem \ref{FES}. By \cite[Theorem 1,
 page 87]{Ra}, there exist \'etale extensions $A_i|\cO_L$ and maximal ideals
 $m_i$ of $A_i$ such that $\cO_{L^h}=\displaystyle\lim_{\rightarrow}(A_i)_{m_i}$.  
 We have the exact sequences
 $$
 \Omega_{\cO_L|\cO_K}\otimes_{\cO_L}A_i\stackrel{\alpha}{\rightarrow} \Omega_{A_i|\cO_K}
 \rightarrow \Omega_{A_i|\cO_L}\rightarrow 0
 $$
 of Theorem \ref{FES}. Since $A_i|\cO_L$ is \'etale, we have that this 
 map is formally \'etale (\cite[Definition IV.17.3.1]{EGA}) and is thus formally 
 unramified and formally smooth (\cite[Definition IV.17.1.1]{EGA}). Thus 
 $\Omega_{A_i|\cO_L}=0$ by 
 \cite[Proposition IV.17.2.1]{EGA} and $\alpha$ is injective by \cite[Proposition 
 IV.17.2.3]{EGA}. By this calculation and \cite[Proposition~16.9]{Eis},
 \begin{equation}\label{eqRN2}
 \Omega_{\cO_L|\cO_K}\otimes_{\cO_L}(A_i)_{m_i}\cong (\Omega_{A_i|\cO_K})\otimes_{A_i}
 (A_i)_{m_i}\cong \Omega_{(A_i)_{m_i}|\cO_K}.
 \end{equation}
 By Theorem 16.8 \cite{Eis} and  equations (\ref{eqRN1}) and (\ref{eqRN2}),
\begin{equation}                               \label{KahCalc32}
\begin{array}{lll}
\Omega_{\cO_L^h|\cO_K^h}&\cong& \Omega_{\cO_{L^h}|\cO_K}\cong
\displaystyle\lim_{\rightarrow}[(\Omega_{(A_i)_{m_i}|\cO_K})\otimes_
{(A_i)_{m_i}}\cO_{L^h}]\\
 & \cong &\displaystyle\lim_{\rightarrow}[(\Omega_{\cO_L|\cO_K})
\otimes_{\cO_L}\cO_{L^h}]\cong (\Omega_{\cO_L|\cO_K})\otimes_{\cO_L}\cO_{L^h}.
\end{array}
\end{equation}
\end{proof}
  
%
%
\subsection{Proofs of Theorems \ref{GalTow} and \ref{Kahler}}  \label{sectpfKahler}
\mbox{ }\sn
We first prove Theorem \ref{GalTow}.

The natural sequence of $\cO_M$-modules
\begin{equation}\label{eqGT1}
0\rightarrow \Omega_{\cO_L|\cO_K}\otimes_{\cO_L}\cO_M\rightarrow \Omega_{\cO_M|\cO_K}
\rightarrow \Omega_{\cO_M|\cO_L}\rightarrow 0
\end{equation}
computed from the extensions of rings $\cO_K\subset \cO_L\subset \cO_M$
is right exact (but the first map might not be injective) by Theorem \ref{FES}.
Tensor this sequence with $\cO_M^h$ over $\cO_M$ to get a right exact sequence of 
$\cO_M^h$-modules
\begin{equation}\label{eqGT2}
0\rightarrow (\Omega_{\cO_L|\cO_K}\otimes_{\cO_L}\cO_M)\otimes_{\cO_M}\cO_M^h\rightarrow
\Omega_{\cO_M|\cO_K}\otimes_{\cO_M}\cO_M^h\rightarrow \Omega_{\cO_M|\cO_L}\otimes_{\cO_M}
\cO_M^h\rightarrow 0.
\end{equation}
By (\ref{KahCalc33}), we have isomorphisms
$$
\Omega_{\cO_M|\cO_L}\otimes_{\cO_M}\cO_M^h\cong \Omega_{\cO_M^h|\cO_L^h},\,\,
\Omega_{\cO_M|\cO_K}\otimes_{\cO_M}\cO_M^h\cong \Omega_{\cO_M^h|\cO_K^h}
$$
and
$$
\begin{array}{l}
(\Omega_{\cO_L|\cO_K}\otimes_{\cO_L}\cO_M)\otimes_{\cO_M}\cO_M^h
\cong \Omega_{\cO_L|\cO_K}\otimes_{\cO_L}\cO_M^h\\
\cong (\Omega_{\cO_L|\cO_K}\otimes_{\cO_L}\cO_L^h)\otimes_{\cO_L^h}\cO_M^h\cong
\Omega_{\cO_{L^h}|\cO_{K^h}}\otimes_{\cO_L^h}\cO_M^h.
\end{array}
$$
Thus (\ref{eqGT2}) is the right exact sequence
\begin{equation}\label{eqGT3}
0\rightarrow \Omega_{\cO_{L^h}|\cO_{K^h}}\otimes_{\cO_{L^h}}\cO_{M^h}\rightarrow
\Omega_{\cO_{M^h}|\cO_{K^h}}\rightarrow \Omega_{\cO_{M^h}|\cO_{L^h}}\rightarrow 0
\end{equation}
of  Theorem \ref{FES}. Since $\cO_M^h$ is a faithfully flat $\cO_M$-module, we 
have that
(\ref{eqGT1}) is exact if and only if (\ref{eqGT3}) is exact.

By assumption, $L|K$ and $M|L$ are towers of Galois extensions
$$
K=K_0\subset K_1\subset \cdots\subset K_r=L\mbox{ and }
L=L_0\subset L_1\subset \cdots \subset L_s=M
$$
so 
$$
K^h=K_0^h\subset K_1^h\subset \cdots \subset K_r^h=L^h\mbox{ and }
L^h=L_0^h\subset L_1^h\subset \cdots \subset L_s^h=M^h
$$
are towers of Galois extensions. Since each $K_{i+1}^h|K_i^h$ is unibranched,
there exist factorizations
$$
K_i^h\subset U_i^1\subset U_i^2\subset \cdots \subset U_i^{m_i}=K_{i+1}^h
$$
where $U_i^1$ is the inertia field of $K_{i+1}^h|K_i^h$ and $U_i^{j+1}|U_i^j$ is Galois 
of prime degree. These extensions are all necessarily unibranched, so $U_i^1|K_i^h$ 
satisfies 2) of Lemma \ref{KahCalc9} and $U_i^{j+1}|U_i^j$ satisfies 1) of Lemma
\ref{KahCalc9} for $1\le j$. Similarly, we have
factorizations
$$
L_i^h\subset V_i^1\subset V_i^2\subset \cdots \subset V_i^{n_i}=L_{i+1}^h
$$
where  $V_i^1|L_i^h$ satisfies 2) of Lemma \ref{KahCalc9} and $V_i^{j+1}|V_i^j$ satisfies 1) 
of Lemma \ref{KahCalc9} for $1\le j$. By Lemma \ref{KahCalc9}, we have exact sequences
$$
\begin{array}{c}
0\rightarrow \Omega_{\cO_{U^1_0}|\cO_{K^h}}\otimes_{\cO_{U_0^1}}\cO_{U_0^2}
\rightarrow \Omega_{\cO_{U_0^2}|\cO_{K^h}}\rightarrow \Omega_{\cO_{U_0^2}|\cO_{U_0^1}}
\rightarrow 0\\
0\rightarrow \Omega_{\cO_{U^2_0}|\cO_{K^h}}\otimes_{\cO_{U_0^2}}\cO_{U_0^3}
\rightarrow \Omega_{\cO_{U_0^3}|\cO_{K^h}}\rightarrow \Omega_{\cO_{U_0^3}|\cO_{U_0^2}}
\rightarrow 0\\
\vdots\\
0\rightarrow \Omega_{\cO_{K_1^h}|\cO_{K^h}}\otimes_{\cO_{K_1^h}}\cO_{U_1^1}
\rightarrow \Omega_{\cO_{U_1^1}|\cO_{K^h}}\rightarrow \Omega_{\cO_{U_1^1}|\cO_{K_1^h}}
\rightarrow 0\\
0\rightarrow \Omega_{\cO_{U^1_1}|\cO_{K^h}}\otimes_{\cO_{U_1^1}}\cO_{U_1^2}
\rightarrow \Omega_{\cO_{U_1^2}|\cO_{K^h}}\rightarrow \Omega_{\cO_{U_1^2}|\cO_{U_1^1}}
\rightarrow 0\\
\vdots\\
0\rightarrow \Omega_{\cO_{L^h}|\cO_{K^h}}\otimes_{\cO_{L^h}}\cO_{V_0^1}
\rightarrow \Omega_{\cO_{V_0^1}|\cO_{K^h}}\rightarrow \Omega_{\cO_{V_0^1}|\cO_{L^h}}
\rightarrow 0\\
0\rightarrow \Omega_{\cO_{V_0^1}|\cO_{K^h}}\otimes_{\cO_{V_0^1}}\cO_{V_0^2}
\rightarrow \Omega_{\cO_{V_0^2}|\cO_{K^h}}\rightarrow \Omega_{\cO_{V_0^2}|\cO_{V_0^1}}
\rightarrow 0\\
\vdots\\
0\rightarrow \Omega_{\cO_{L_1^h}|\cO_{K^h}}\otimes_{\cO_{L_1^h}}\cO_{V_1^1}
\rightarrow \Omega_{\cO_{V_1^1}|\cO_{K^h}}\rightarrow \Omega_{\cO_{V_1^1}|\cO_{L_1^h}}
\rightarrow 0\\
0\rightarrow \Omega_{\cO_{V_1^1}|\cO_{K^h}}\otimes_{\cO_{V_1^1}}\cO_{V_1^2}
\rightarrow \Omega_{\cO_{V_1^2}|\cO_{K^h}}\rightarrow \Omega_{\cO_{V_1^2}|\cO_{V_1^1}}
\rightarrow 0\\
\vdots\\
0\rightarrow \Omega_{\cO_{V_{s-1}^{n_s-1}}|\cO_{K^h}}\otimes_{\cO_{V_{s-1}^{n_s-1}}}\cO_{M^h}
\rightarrow \Omega_{\cO_{M^h}|\cO_{K^h}}\rightarrow \Omega_{\cO_{M^h}|\cO_{V_{s-1}^{n_s-1}}}
\rightarrow 0.\\
\end{array}
$$
In particular, differentiation defines an injection of $\cO_{V_0^1}$-modules
$$
\Omega_{\cO_{L^h}|\cO_{K^h}}\otimes_{\cO_{L^h}\cO_{V_0^1}}\rightarrow 
\Omega_{\cO_{V_0^1|\cO_K^h}}.
$$
Since $\cO_{V_0^2}$ is a flat $\cO_{V_0^1}$-module, we have injections
$$
\Omega_{\cO_{L^h}|\cO_{K^h}}\otimes_{\cO_{L^h}}\cO_{V_0^2}\cong
(\Omega_{\cO_{L^h}|\cO_{K^h}}\otimes_{\cO_{L^h}}\cO_{V_0^1})\otimes_{\cO_{V_0^1}}\cO_{V_0^2}
\rightarrow \Omega_{\cO_{V_0^1}|\cO_{K^h}}\otimes_{\cO_{V_0^1}}\cO_{V_0^2}
\rightarrow \Omega_{\cO_{V_0^2}|\cO_K^h}
$$
and continuing, we obtain that differentiation gives an injection of $\cO_{M^h}$-modules
$$
\Omega_{\cO_{L^h}|\cO_{K^h}}\otimes_{\cO_{L^h}}\cO_{M^h}\rightarrow
\Omega_{\cO_{M^h}|\cO_{K^h}}
$$
so that (\ref{eqGT3}) is short exact and thus (\ref{eqGT1}) is short exact. 

Since $\cO_M$ is a faithfully flat $\cO_L$, module, we have that $\Omega_{\cO_L|\cO_K}
\otimes_{\cO_L}\cO_M=0$ if and only if $\Omega_{\cO_L|\cO_K}=0$, and so 
$\Omega_{\cO_M|\cO_K}=0$ if and only if $\Omega_{\cO_M|\cO_L}=0$ and 
$\Omega_{\cO_L|\cO_K}=0$.
\qed

\parm
We now prove Theorem \ref{Kahler}.
We first prove Statement 1). By \cite[Theorem 16.8]{Eis}, we have an  isomorphism of 
$\cO_{K\sep}$-modules
$$
\Omega_{\cO_{K\sep}|\cO_K}
\cong\lim_{\rightarrow}[(\Omega_{\cO_L|\cO_K})\otimes_{\cO_L}\cO_{K\sep}].
$$
where the limit is over finite Galois subextensions $L|K$ of $K\sep$. 

If $\Omega_{\cO_L|\cO_K}=0$ for all finite Galois subextensions of $K^{\rm sep}$, then 
it follows immediately from the above formula that $\Omega_{\cO_{K^{\rm sep}}|\cO_K}=0$.

Assume that $\Omega_{\cO_{K^{\rm sep}}|\cO_K}=0$ and $L|K$ is a finite Galois 
subextension of $K^{\rm sep}$. If $\Omega_{\cO_L|\cO_K}\ne 0$, then there exists $0\ne x\in
\Omega_{\cO_L|\cO_K}$ and a 
finite Galois extension $N$ of $K$ such that $N$ contains $L$ and the image of $x\otimes 1$ 
by the natural homomorphism
$$
(\Omega_{\cO_L|\cO_K})\otimes_{\cO_L}\cO_{K^{\rm sep}}\rightarrow  
(\Omega_{\cO_N|\cO_K})\otimes_{\cO_N}\cO_{K^{\rm sep}}
$$
 is zero. 
Since $\cO_{K^{\rm sep}}$ is a faithfully flat $\cO_N$-module 
(by Lemma~\ref{KahCalc36}) we have that the image of $x\otimes 1$ by the natural homomorphism
$$
\Omega_{\cO_L|\cO_K}
\otimes_{\cO_L}\cO_N\rightarrow \Omega_{\cO_N|\cO_K}
$$
 is zero, so that $x\otimes 1=0$ by Theorem \ref{GalTow}. Thus $x=0$ since 
$\cO_N$ is a faithfully flat $\cO_L$-module, giving a contradiction, and showing that 
$\Omega_{\cO_L|\cO_K}=0$.

We now prove Statement 2). We have that  $\Omega_{\cO_{L}|\cO_K}\cong \Omega_{\cO_{L}
|\cO_{K\ina}}$ by  Proposition~\ref{KahCalc20*}.   For $0\le i\le\ell-1$, 
$\Omega_{\cO_{K_i}|\cO_{K_0}}=0$ if and only if $\Omega_{\cO_{K_i}|\cO_{K_0}}
\otimes_{\cO_{K_i}}\cO_{K_{i+1}}=0$  since 
$\cO_{K_{i+1}}$ is a faithfully flat $\cO_{K_i}$-module by  Lemma \ref{KahCalc36}. 
Statement 2) now follows from Lemma \ref{KahCalc9} by induction on $i$ in equation
(\ref{KahCalc8}).
\qed

%

%
%
\section{Proof of Theorems~\ref{GRthm} and~\ref{GRthm+}}   \label{sectprf}
Take a valued field $(K,v)$ and extend $v$ to the separable closure $K\sep$ of $K$.
Recall that we call $(K,v)$ a deeply ramified field if it satisfies (DRvg) and (DRvr). 

Throughout we assume that $\chara Kv=p>0$. If $\chara K=0$, then we set $K':=
K(\zeta_p)$ with $\zeta_p$ a primitive $p$-th root of unity and extend $v$ to $K'$.
If $\chara K=p$, then we set $K':=K$. The next proposition will show that in our proof 
of Theorem~\ref{GRthm} we can assume that $K=K'$.
\begin{proposition}                            \label{K'}
1) If $\Omega_{\cO_{K\sep}|\cO_K}=0$, then $\Omega_{\cO_L|\cO_{K'}}=0$ holds for 
every finite Galois extension $(L|K',v)$.
\sn
2) If $(K',v)$ is a deeply ramified field, then so is $(K,v)$.
\end{proposition}
\begin{proof}
1): Assume that $\Omega_{\cO_{K\sep}|\cO_K}=0$. By part 1) of Theorem~\ref{Kahler} 
this implies that $\Omega_{\cO_L|\cO_K}=0$ for every finite Galois extension $(L|K,v)$. 
In all cases, $(K'|K,v)$ is a finite Galois extension, possibly trivial. 
Take any finite Galois extension 
$(L|K',v)$, let $N$ be the normal hull of $L|K$, and take any extension of $v$ to $N$. 
Then $(N|K,v)$ is a finite Galois extension, so we have $\Omega_{\cO_N|\cO_K}=0$. Since
also $(N|K',v)$ and $(K'|K,v)$ are finite Galois extensions, Theorem~\ref{GalTow} shows 
that $\Omega_{\cO_N|\cO_{K'}}=0$. Finally, since $(N|L,v)$ and $(L|K',v)$ are finite Galois
extensions, Theorem~\ref{GalTow} shows that $\Omega_{\cO_L|\cO_{K'}}=0$.
%
\sn
2): This follows from \cite[Theorem~1.8]{KuRz}.
\end{proof}

We split Theorem~\ref{GRthm} into the following two propositions, which 
we will prove separately. In view of Proposition~\ref{K'} it suffices to prove them 
under the assumption that $K$ contains a primitive $p$-th root of unity if 
$\chara K=p>0$, i.e., $K=K'$.
\begin{proposition}               \label{GRthm1}
If $\Omega_{\cO_{K\sep}|\cO_K}=0$, then $(K,v)$ is a deeply ramified field. 
\end{proposition}

\begin{proposition}               \label{GRthm2}
If $(K,v)$ is a deeply ramified field, then $\Omega_{\cO_{K\sep}|\cO_K}=0$. 
\end{proposition}

One of the implications of Theorem~\ref{GRthm+} will be proved in
Proposition~\ref{propGRthm1}, and the other in Proposition~\ref{GEofdrf}.

\sn
%
%
\subsection{Proof of Proposition~\ref{GRthm1}}                
\mbox{ }\sn
We will need some preparations. 
If the valued field $(K,v)$ is of characteristic 0 with residue characteristic $p>0$,
then we decompose $v=v_0\circ v_p\circ \ovl{v}$, where $v_0$ is the finest coarsening 
of $v$ that has residue characteristic 0, $v_p$ is a rank 1 valuation on $Kv_0\,$, and 
$\ovl{v}$ is the valuation induced by $v$ on the residue field of $v_p$ (which is of
characteristic $p>0$). The valuations $v_0$ and $\ovl{v}$ may be trivial. Note that 
while it makes no sense to compose the valuations as functions, in this notation the 
valuations are interpreted as their associated places (as we have done before by writing
``$Kv$''): in this way, $Kv=K(v_0\circ v_p\circ \ovl{v})=((Kv_0)v_p)\ovl{v}$. For 
simplicity, we will write $v_0v_p$ for $v_0\circ v_p$ and $v_p \ovl{v}$ for 
$v_p\circ \ovl{v}$.
In our decomposition, the valuation $v_p$ is at the center, so we define 
$\crf (K,v):= (Kv_0)v_p$ as one may call it the ``central residue field''. 
In the equal characteristic case, we set $\crf(K,v):=Kv$.

\pars
Now take any valued field $(K,v)$ of residue characteristic $p>0$. 
We will use the following observation; we note that $\cC_{vK}(vp)$ was denoted by 
$(vK)_{vp}$ in \cite{KuRz}.
\begin{proposition}                   \label{condDRvr}
If $K=K'$ and $\cC_{vK}(vp)$ is $p$-divisible, $Kv$ is perfect and all Galois 
extensions $(L|K,v)$ of prime degree $p$ with nontrivial defect satisfy 
$\Omega_{\cO_L|\cO_K}=0$, then $(K,v)$ satisfies (DRvr). 
\end{proposition}
\begin{proof}
We will show that the assumptions imply that $\crf(K,v)$ is perfect. Then the assertion
follows from \cite[Proposition~4.13]{KuRz} since by \cite[Theorem~1.4]{CuKuRz}, all 
Galois extensions $(L|K,v)$ of prime degree $p$ with nontrivial defect that satisfy 
$\Omega_{\cO_L|\cO_K}=0$ have independent defect in the sense of \cite{KuRz,CuKuRz}.

In the equal characteristic case, $\crf(K,v)=Kv$ and there is nothing to show. So we
assume that $(K,v)$ has mixed characteristic. 
%
%
Take any nonzero element of $\crf(K,v)$; it can be written as $bv_0 v_p$ 
with $b\in K$. 
Consider the extension $K(\eta)|K$ with $\eta^p=b$. We have that  
$\eta v_0 v_p$ is a $p$-th root of $bv_0 v_p$ in $\crf(K(\eta),v)$. 

Suppose that $bv_0 v_p$ does not have a $p$-th root in $\crf(K,v)$, so $K(\eta)|K$ is 
a Kummer extension of degree $p$. Then $(K(\eta)v_0v_p|Kv_0 v_p$ is purely inseparable 
of degree $p$. It follows that $v_0 v_p K(\eta)=v_0 v_p K$ and that $(K(\eta)|K,
v_0 v_p)$ and $(K(\eta)v_0v_p|Kv_0 v_p, \bar v)$ are unibranched. Consequently, 
$(K(\eta)|K,v)$ is unibranched. Further, as $\cC_{vK}(vp)$ and thus also
$\ovl{v}(Kv_0 v_p)$ is $p$-divisible, we have $\ovl{v}(K(\eta)v_0 v_p)=\ovl{v}
(Kv_0 v_p)$ and therefore, $vK(\eta)=vK$. Moreover, $K(\eta)v=K(\eta)v_0v_p\bar v$ is
a purely inseparable extension of $Kv=Kv_0v_p\bar v$ and since $Kv$ is perfect, we 
find that $K(\eta)v=Kv$. Thus $(K(\eta)|K,v)$ is
an extension with nontrivial defect. Since $(K,v)$ is an independent defect 
field, the defect must be independent. Hence by \cite[condition b) of Theorem~1.8]{CuKuRz},
\[
v(b-K^p)\>=\>\frac{p}{p-1}vp-\{\alpha\in pvK\mid \alpha >H\} 
\]
for some convex subgroup $H$ of $vK$ that does not contain $vp$, so also does not
contain $\frac{p}{p-1}vp$. It follows that there is some $a\in K$ such that
$v(b-a^p)>vp$, whence $(b-a^p)v_0 v_p=0$. This shows that $(a v_0 v_p)^p
=b v_0 v_p$, so that $bv_0 v_p$ has a $p$-th root in 
$\crf(K,v)$, which contradicts our assumption. 

We have now proved that $\crf(K,v)$ is perfect, as desired.
\end{proof}

\parm
Now we are ready to prove one part of Theorem~\ref{GRthm+}:
%
%
\begin{proposition}                       \label{propGRthm1}
If $K=K'$ and if $\Omega_{\cO_L|\cO_K}=0$ for all
unibranched Galois extensions $(L|K,v)$ of prime degree $p$, then $(K,v)$ is 
a deeply ramified field. 
\end{proposition}
\begin{proof}
We first deal with the equal characteristic case. In this case, $\cC_{vK}(vp)=vK$. 
Suppose that $vK$ is not $p$-divisible and take some $a\in K$ such that $va\notin
pvK$. We may assume that $va<0$. Take $\vartheta\in K\sep$ such that $\vartheta^p
-\vartheta=a$. Then $pv\vartheta=va$ and $(K(\vartheta)|K,v)$ is an Artin-Schreier
extension with $\rme(K(\vartheta)|K,v)=p$. Hence by Theorem~\ref{OASe}, 
$\Omega_{\cO_{K(\vartheta)}|\cO_K}\ne 0$, contradiction. Thus $\cC_{vK}(vp)=vK$ is
$p$-divisible, and in particular, (DRvg) holds.

Suppose that $Kv$ is not perfect, and take $b\in\cO_K^\times$ such that $bv$ does not 
have a $p$-th root in $Kv$. Take $c\in K$ such that $vc<0$ and $\vartheta\in 
K\sep$ such that $\vartheta^p-\vartheta=c^pb$. Then $(K(\vartheta)|K,v)$ is an 
Artin-Schreier extension with $K(\vartheta)v=Kv(bv^{1/p})$. Hence by Theorem~\ref{OASf}, 
$\Omega_{\cO_{K(\vartheta)}|\cO_K}\ne 0$, which again is a contradiction. Hence $Kv$ is 
perfect. Now Proposition~\ref{condDRvr} shows that also (DRvr) holds and consequently,
$(K,v)$ is a deeply ramified field. 

\parm
Now we deal with the mixed characteristic case. If we are able to show that $(K,v)$
satisfies (DRvg), $\cC_{vK}(vp)$ is $p$-divisible and $Kv$ is perfect, then we can as 
before apply Proposition~\ref{condDRvr} to obtain again that $(K,v)$ is a deeply ramified
field. 

Suppose that there is an archimedean component of $vK$ which is discrete. Pick $a
\in K$ such that $va<0$ and $va+\cC_{vK}^+(va)$ is the largest negative
element in $\cA_{vK}(va)$. Take $\eta\in K\sep$ such that $\eta^p\in K$ with 
$v\eta^p=va$. Then $v\eta+\cC_{vL}^+(v\eta)$ is the largest negative element in 
$\cA_{vL}(v\eta)$, not contained in $\cA_{vK}(pv\eta)$, and $(K(\eta)|K,v)$ is a 
Kummer extension with $\rme(K(\eta)|K,v)=p$. It follows that 
$(vK(\eta)/\cC_{vL}^+(v\eta):vK/\cC_{vK}^+(pv\eta))=p$, hence we must have 
$\cC_{vL}^+(v\eta)=\cC_{vK}^+(pv\eta)$. Therefore, $\cE$ is of type (DL2c) with
$H_\cE=\cC_{vL}^+(v\eta)$, so $\cM_\cE$ is a
principal $\cO_\cE$-ideal. From case i) of Theorem~\ref{OKume} we now infer that 
$\Omega_{\cO_{K(\eta)}|\cO_K}\ne 0$, contradiction. 

Suppose that $\cC_{vK}(vp)$ is not $p$-divisible and take some $a\in K$ such that $va\in
\cC_{vK}(vp)\setminus p\,\cC_{vK}(vp)\,$. We may assume that $va<0$. Take $\eta\in K\sep$ such
that $\eta^p=a$. Then $pv\eta=va$ and $(K(\eta)|K,v)$ is a Kummer extension 
with $\rme(K(\eta)|K,v)=p$. We have that $vI_\eta\cap\cC_{vL}(vp)\ne\emptyset$. This
implies that $vp\notin H_\cE\,$, whence $p\in\cM_\cE\,$. Again from case i) of
Theorem~\ref{OKume} we conclude that $\Omega_{\cO_{K(\eta)}|\cO_K}\ne 0$, contradiction. 

Suppose that $Kv$ is not perfect, and take $b\in\cO_K^\times$ such that $bv$ does not 
have a $p$-th root in $Kv$. Take $\eta\in K\sep$ such that $\eta^p=b$. Then $(K(\eta)|
K,v)$ is a Kummer extension with $K(\eta)v=Kv(bv^{1/p})$. Hence by Theorem~\ref{OKumf}, 
$\Omega_{\cO_{K(\eta)}|\cO_K}\ne 0$, which is again a contradiction.
This finishes the proof that $(K,v)$ is deeply ramified. 
\end{proof}

Now Proposition~\ref{GRthm1} follows from Proposition~\ref{propGRthm1} in conjunction with
part 1) of Proposition~\ref{K'}.

\sn
%
%
\subsection{Proof of Proposition~\ref{GRthm2}}                
\mbox{ }\sn
We first observe:
\begin{proposition}                          \label{GEofdrf}
Take a deeply ramified field $(K,v)$ such that $K=K'$, and a unibranched 
Galois extension $(L|K,v)$ of prime degree. Then $\Omega_{\cO_L|\cO_K}=0$.
\end{proposition}
\begin{proof}
In view of Theorem~\ref{MTsdrf1}, we only have to deal with the case of defectless
extensions. 

\pars
Assume that $\chara K=p$ and $(L|K,v)$ is an Artin-Schreier extension of degree $p$.
We have that $vK$ is $p$-divisible and $Kv$ is perfect by \cite[Lemma~4.2]{KuRz}.
Thus, the case of $\rme(L|K)=p$ cannot appear and we must have that $\rmf(L|K)=p$
with the extension $Lv|Kv$ separable. Hence $\Omega_{\cO_L|\cO_K}=0$ by 
Theorem~\ref{seprfe}. 

\pars
Assume that $(L|K,v)$ is a Kummer extension of prime degree $q=\rmf(L|K)$.
Again, $Lv|Kv$ is separable, so $\Omega_{\cO_L|\cO_K}=0$ by Theorem~\ref{seprfe}. 

\pars
Finally, assume that $\cE=(L|K,v)$ is a Kummer extension of prime degree $q=\rme(L|K)$.
Since each archimedean component of the deeply ramified field $(K,v)$ is dense, the same
holds for all archimedean components of $vL$. This shows that $\cE$ is not of type
(DL2c), so $\cM_\cE$ is a nonprincipal $\cO_\cE$-ideal.

If $q\ne \chara Kv$, then $vq=0$ implies that $q\notin\cM_L$ and hence $q\notin\cM_\cE\,$. 
From case i) of Theorem~\ref{OKume} we now obtain that $\Omega_{\cO_L|\cO_K}=0$.

If $q=\chara Kv$, then necessarily $\chara K=0$. By \cite[part (1) of Lemma~4.3]{KuRz}, 
$\cC_{vK}(vq)$ is $q$-divisible. If case ii) of Theorem~\ref{OKume} would apply, then 
by (\ref{vz-1}), $0<v(\eta-1)<v(\zeta_q-1)=\frac{vq}{q-1}$ with $v(\eta-1)\notin 
vK$, whence $v(\eta-1) \in \cC_{vL}(vq)$ and $(\cC_{vL}(vq):\cC_{vK}(vq))=q$. As this
contradicts the fact that $\cC_{vK}(vq)$ is $q$-divisible, case ii) cannot appear and
moreover, $vq\in H_\cE$ and thus $q\notin \cM_\cE$ since $\cC_{vL}(vq)=\cC_{vK}(vq)$.
By case i) of Theorem~\ref{OKume} we conclude that $\Omega_{\cO_L|\cO_K}=0$.
\end{proof}

\pars
Take any deeply ramified field $(K,v)$. By \cite[Corollary 1.7 (2)]{KuRz}, also the
henselization $(K,v)^h$ of $(K,v)$ inside of $(K\sep,v)$, for any of the conjugate extensions 
from $v$ from $K$ to $K\sep$, is a deeply ramified field.
By Proposition~\ref{KahCalc17} it suffices to prove that $\Omega_{\cO_{K\sep}|
\cO_{K^h}}=0$. We may therefore assume from the start that $(K,v)$ is henselian.
 
Part 1) of Theorem~\ref{Kahler} shows that in order to prove that $\Omega_{\cO_{K\sep}
|\cO_K}=0$ it suffices to prove that $\Omega_{\cO_L|\cO_K}=0$ for all finite Galois
subextensions $(L|K,v)$ of $(K\sep|K,v)$. Proposition~\ref{towerprop} shows that after
enlarging $(L|K,v)$ to a finite Galois extension $(M|K,v)$ if necessary, 
there is a tower of field extensions
\[
K\,\subset\, M_0\,\subset\, M_1\,\subset \cdots\subset\, M_m\>=\> M
\]
where $M_0$ is the inertia field of $(M|K,v)$ and each extension $M_{i+1}|
M_i$ is a Kummer extension of prime degree, or an Artin-Schreier extension if the
extension is of degree $p=\chara K$. By part 2) of Theorem~\ref{Kahler}, to prove that 
$\Omega_{\cO_M|\cO_K}=0$ it suffices to prove that $\Omega_{\cO_{M_{i+1}}|\cO_{M_i}}
=0$ for $0\le i\le m-1$. By Theorem~\ref{MTsdrf2}, $(M_i,v)$ is a deeply 
ramified field for each $i$, hence $\Omega_{\cO_{M_{i+1}}|\cO_{M_i}}=0$ by
Proposition~\ref{GEofdrf}. We have shown that $\Omega_{\cO_M|\cO_K}=0$. 

Since $M|K$ is a Galois extension, so is $M|L$. Hence we can apply 
Theorem~\ref{GalTow} to conclude that $\Omega_{\cO_L|\cO_K}=0$.
This completes our proof of Theorem~\ref{GRthm}.

\sn


\begin{thebibliography}{99}




\bibitem{AG} Cutkosky, S.D.: {\it Introduction to Algebraic Geometry}, Graduate Studies 
in Mathematics, 188. American Mathematical Society, Providence, R.I., 2018


\bibitem{CuKuRz} Cutkosky, S.D.\ -- Kuhlmann, F.-V.\ -- Rzepka, A.: 
{\it Characterizations of Galois extensions with independent defect} (2023), to appear
in: Mathematische Nachrichten; arXiv:2305.10022

\bibitem{CuNo} Cutkosky, S.D.\ -- Novacoski J.: {\it Essentially finite generation of
valuation rings in terms of classical invariants}, Math.\ Nachrichten {\bf 294} (2021), 
15--37


\bibitem{D} Datta, R.: {\it Essential finite generation of extensions of valuation 
rings}, Math.\ Nachrichten {\bf 296} (2021), 1041--1055

\bibitem{Eis} Eisenbud, D.: {\it Commutative Algebra with a view toward Algebraic 
Geometry}, Springer-Verlag, New York, 1995

\bibitem{En} Endler, O.: {\it Valuation theory}, Springer-Verlag, Berlin, 1972


\bibitem{GR} Gabber, O.\ -- Ramero, L.: {\it Almost ring theory}, Lecture Notes in 
Mathematics {\bf 1800}, Springer-Verlag, Berlin, 2003

\bibitem{EGA} Grothendieck, A.\ -- Dieudonn\'e. J.: {\it \'El\'ements de G\'eom\'etrie
Alg\'ebrique, IV \'Etude locale des sch\'emas et des morphismes de sch\'emas, 
quatri\'eme partie}, Pub.\ Math.\ IHES 32 (1967)

\bibitem{JK} Jahnke, F.\ -- Kartas, K.: {\it Beyond the Fontaine-Wintenberger theorem},
arXiv:2304.05881




\bibitem{Ku30}  
Kuhlmann, F.-V.: {\it A classification of Artin--Schreier defect extensions and a
characterization of defectless fields}, Illinois J.\ Math.\ {\bf 54} (2010), 397--448





\bibitem{KuTI} Kuhlmann, F.-V.: {\it Topics in higher ramification theory I: ramification
ideals}, in preparation; available on
\newline
{\tt https://fvkuhlmann.de/Fvkprepr.html}

\bibitem{KuTII} Kuhlmann, F.-V.\ -- Rzepka, A.: {\it Topics in higher ramification theory
II}, in preparation

\bibitem{KuRz} Kuhlmann, F.-V.\ -- Rzepka, A.: {\it The valuation theory of deeply
ramified fields and its connection with defect extensions}, Transactions Amer.\ Math.\
Soc.\ {\bf 376} (2023), 2693--2738

\bibitem{KuVl} Kuhlmann, F.-V.\ -- Vlahu, I.: {\it The relative
approximation degree}, Math.\ Z.\ {\bf 276} (2014), 203--235


\bibitem{[L]} Lang, S.: {\it Algebra}, revised 3rd.\ ed., Springer-Verlag, 
New York, 2002

\bibitem{Mat} Matsumura, H.: {\it Commutative Ring Theory},
Cambridge Univ. Press, Cambridge UK, 1989

\bibitem{Na} Nagata, M.: {\it Local Rings}, Interscience Publishers, New York, 
London, 1962

\bibitem{NN1} Nart, E.\ -- Novacoski, J.H.: {\it The defect formula}, Adv.\ Math.\
{\bf 428} (2023), Paper No.\ 109153, 44 pp

\bibitem{NN2} Nart, E.\ -- Novacoski, J.H.: {\it Minimal limit key polynomials}
(2023), arXiv:2311.13558

\bibitem{NE} Neukirch, N.: {\it Algebraic Number Theory}, Berlin Springer Verlag, 
Heidelberg,  1999.

\bibitem{No} Novacoski, J.: {\it Generators for extensions of valuation rings} (2024),
submitted; arXiv:2401.00182  

\bibitem{NS} Novacoski, J.\ -- Spivakovsky, M.: {\it Kähler differentials, pure 
extensions and minimal key polynomials} (2023), submitted; arXiv:2311.14322 

\bibitem{Ra} Raynaud, M.: {\it Anneaux Locaux Hens\'eliens}, Lecture Notes in 
Mathematics {\bf 169}, Springer-Verlag, Berlin Heidelberg New York, 1970

\bibitem{Rot} Rotman, J.: {\it An Introduction to Homological Algebra}, Pure and 
Applied Mathematics, 85. Academic Press, Inc., New York-London, 1979

\bibitem{Rot2} Rotman, J.: {\it An introduction to homological algebra}, Second 
edition. Universitext. Springer, New York, 2009

\bibitem{Sch} Scholze, P.: {\it Perfectoid spaces}, Publ.\ Math.\ Inst.\ Hautes 
\'Etudes Sci.\ {\bf 116} (2012), 245--313

\bibitem{stackex} Contributor on StackExchange: 
{\tt https://math.stackexchange.com/questions/403924/
xp-c-has-no-root-in-a-field-f-if-and-only-if-xp-c-is-irreducible}


\bibitem{HTT} Tang, H.T.: {\it Gauss' lemma}, Proc.\ Amer.\ Math.\ Soc.\ {\bf 35} 
(1972), 372--376

\bibitem{Th1} Thatte, V.: {\it Ramification theory for Artin-Schreier extensions of
valuation rings}, J.\ Algebra {\bf 456} (2016), 355--389

\bibitem{Th2} Thatte, V.: {\it Ramification theory for degree $p$ extensions of 
valuation rings in mixed characteristic (0,p)}, J.\ Algebra {\bf 507} (2018), 225--248


\bibitem{ZSI} Zariski, O.\ -- Samuel, P.: {\it Commutative Algebra}, Vol.\ I, D. Van
Nostrand, Princeton N.J., 1958

\bibitem{ZS2} Zariski, O.\ -- Samuel, P.: {\it Commutative Algebra}, Vol.\ II, 
New York--Heidelberg--Berlin, 1960

\end{thebibliography}
\end{document}